\documentclass[a4paper,12pt]{amsart}

\usepackage[english]{babel}

\usepackage{amsmath}
\usepackage{amsthm}
\usepackage{amssymb}
\usepackage{amsfonts}
\usepackage{mathtools}
\usepackage{extarrows}
\usepackage{bbm}

\usepackage[
    a4paper,
    left=3cm,
    right=3cm,
    top=2.5cm,
    bottom=3cm
]{geometry}

\usepackage[shortlabels]{enumitem}
\usepackage{multicol}
\usepackage{pdflscape}
\usepackage[normalem]{ulem}
\usepackage{marginnote}
\usepackage{comment}
\usepackage{cancel}

\usepackage[dvipsnames]{xcolor}
\usepackage{graphicx}
\usepackage[export]{adjustbox}
\usepackage{subcaption}

\usepackage[all]{xy}

\usepackage{tikz}

\usetikzlibrary{
    arrows,
    arrows.meta,
    automata,
    backgrounds,
    calc,
    chains,
    decorations.markings,
    fit,
    matrix,
    mindmap,
    patterns,
    positioning,
    scopes,
    shapes,
    shapes.geometric
}

\usepackage{thmtools}
\usepackage{todonotes}

\definecolor{teal}{RGB}{0,128,128}
\definecolor{coral}{RGB}{220,90,70}

\usepackage{hyperref}
	\hypersetup{breaklinks=true,colorlinks=true,
linkcolor=coral,citecolor=teal,
urlcolor=MidnightBlue}

\usepackage{bookmark}

\setlist[enumerate]{
    itemsep=0.2em,
    topsep=0.25em
}

\makeatletter

\renewcommand{\@dotsep}{4.5}

\renewcommand{\l@section}{%
    \@tocline{1}{0.4em}{0em}{2.5em}{\bfseries}%
}

\renewcommand{\l@subsection}{%
    \@tocline{2}{0em}{2.5pc}{5pc}{}%
}

\renewcommand{\l@subsubsection}{%
    \@tocline{3}{0.1em}{4.9em}{5.8em}{}%
}

\renewcommand{\tocsection}[3]{%
    \indentlabel{%
        \@ifnotempty{#2}{%
            \bfseries\ignorespaces#1 #2\quad
        }%
    }%
    \bfseries#3%
}

\renewcommand{\tocsubsection}[3]{%
    \indentlabel{%
        \@ifnotempty{#2}{%
            \ignorespaces#1 #2\quad
        }%
    }%
    #3%
}

\def\@tocline#1#2#3#4#5#6#7{%
    \relax
    \ifnum#1>\c@tocdepth
    \else
        \par
        \addpenalty\@secpenalty
        \addvspace{#2}%
        \begingroup
            \hyphenpenalty\@M
            \@ifempty{#4}{%
                \@tempdima
                \csname r@tocindent\number#1\endcsname
                \relax
            }{%
                \@tempdima#4\relax
            }%
            \parindent\z@
            \leftskip#3\relax
            \advance\leftskip\@tempdima\relax
            \rightskip\@pnumwidth plus1em
            \parfillskip-\@pnumwidth
            #5%
            \leavevmode
            \hskip-\@tempdima
            {#6}%
            \nobreak
            \leaders
            \hbox{%
                $\m@th
                \mkern\@dotsep mu
                \hbox{.}%
                \mkern\@dotsep mu$%
            }%
            \hfill
            \nobreak
            \hbox to\@pnumwidth{%
                \@tocpagenum{%
                    \ifnum#1=1
                        \bfseries
                    \fi
                    #7%
                }%
            }%
            \par
            \nobreak
        \endgroup
    \fi
}

\AtBeginDocument{%
    \expandafter\renewcommand
    \csname r@tocindent0\endcsname{0pt}%
}

\makeatother

\newtheorem{theorem}{Theorem}[section]
\newtheorem{lemma}[theorem]{Lemma}
\newtheorem{fact}[theorem]{Fact}

\newtheorem{proposition}[theorem]{Proposition}

\newtheorem{corollary}[theorem]{Corollary}

\theoremstyle{definition}

\newtheorem{definition}[theorem]{Definition}
\newtheorem{example}[theorem]{Example}

\theoremstyle{remark}

\newtheorem{remark}[theorem]{Remark}

\numberwithin{equation}{section}

\newcommand{\R}{\mathbb{R}}
\newcommand{\C}{\mathbb{C}}
\newcommand{\N}{\mathbb{N}}
\newcommand{\K}{\mathbb{K}}
\newcommand{\T}{\mathbb{T}}
\newcommand{\eps}{\varepsilon}

\newcommand{\pten}{\widetilde{\otimes}_{\pi}}
\newcommand{\polten}[1]{%
    \widetilde{\otimes}_{\pi_s}^{#1, \,s}%
}
\newcommand{\multiten}[1]{%
    \widetilde{\otimes}_{\pi}^{\,#1}%
}
\newcommand{\symmultiten}[1]{%
    \widetilde{\otimes}_{\pi}^{#1,\,s}%
}

\DeclareMathOperator{\sep}{sep}
\DeclareMathOperator{\diam}{diam}

\DeclareMathOperator{\id}{Id}
\DeclareMathOperator{\supp}{supp}
\DeclareMathOperator{\re}{Re}
\DeclareMathOperator{\im}{Im}

\renewcommand{\geq}{\geqslant}
\renewcommand{\leq}{\leqslant}

\begin{document}

\title[Concentration theorems with applications to Kadec-Klee properties]{Concentration theorems for $2$-homogeneous polynomials and bilinear forms with applications to the Kadec-Klee properties}
\dedicatory{ }

\author[S.~Dantas]{Sheldon Dantas}
\address[S.~Dantas]{Czech Technical University in Prague, FEE, Department of Mathematics, Technick\'a 2, 16627, Prague 6, Czech Republic\\
\href{https://orcid.org/0000-0001-8117-3760}{ORCID: \texttt{0000-0001-8117-3760}}}
\email{\texttt{sheldon.dantas@fel.cvut.cz}}
\urladdr{www.sheldondantas.com}

\author[J.~T.~Rodr\'iguez]{Jorge Tom\'as Rodr\'iguez}
\address[J.~T.~Rodr\'iguez]{NuCoMPA, Facultad de Cs. Exactas, Universidad Nacional del Centro de la Provincia de Buenos Aires, (7000) Tandil, Argentina and CONICET.\\
\href{https://orcid.org/0000-0003-4693-2498}{ORCID: \texttt{0000-0003-4693-2498}}}
\email{\texttt{jtrodriguez@nucompa.exa.unicen.edu.ar}}

\begin{abstract}
We study Kadec-Klee properties in spaces of homogeneous polynomials and multilinear forms by presenting optimal results. Indeed, our main tool is a family of concentration theorems showing that, on suitable $p$-convex Banach sequence lattices, a $2$-homogeneous  polynomial or a bilinear form which almost attains its norm at finitely supported vectors is uniformly close to its restriction to the corresponding finite set of coordinates. As a consequence, we obtain the weak-star uniform Kadec-Klee property for spaces of complex $2$-homogeneous polynomials and complex bilinear forms on $p$-convex Banach sequence lattices with constant one for some $p>2$. This applies, in particular, to $c_0$, $\ell_p$, Lorentz spaces $d(w,p)$, Garling spaces $g(w,p)$, among others. In the real setting, a concentration argument based on the modulus of convexity of power type yields further positive results for spaces of operators and bilinear forms, extending previous results in the literature. We also prove that the restriction $p>2$ is optimal for $2$-homogeneous polynomials on $\ell_p$, $d(w,p)$ and $g(w,p)$, and establish general failures of the sequential weak-star Kadec-Klee property for real homogeneous polynomials of degree at least two and for complex homogeneous polynomials and symmetric multilinear forms of degree at least three. These results show that our positive results are optimal respect to the degree and that the assumption $p>2$ is optimal for the classical spaces $\ell_p$, $d(w,p)$ and $g(w,p)$. As a consequence of our results, we establish the strong subdifferentiability of the associated projective and symmetric projective tensor norms.
\end{abstract}

\subjclass[2020]{Primary 46B20, 46G25; Secondary 46B28}
\keywords{Weak-star Kadec--Klee property; homogeneous polynomial; multilinear form; symmetric tensor product; strong subdifferentiability}

\maketitle
\thispagestyle{plain}

\begin{center}
\begin{minipage}{.8\textwidth}
  \centering
\parskip=0ex
    \tableofcontents
\end{minipage}
\end{center}

\newpage
\section{Introduction}

Our main goal in the present paper is to study the weak-star uniform Kadec-Klee and the weak-star Kadec-Klee properties ($w^*$-UKK and $w^*$-KK, respectively) in some dual Banach spaces (see definitions below). These properties play an important role in the geometry of the Banach space (see \cite{HajekTalponen, Lancien, Raja}). Besides their intrinsic interest, they are also closely related to differentiability properties of the norms of the involved spaces (see \cite{DJMR, FP, GodefroyMontesinosZizler, JLPS}). While these properties have been extensively studied for classical Banach spaces and spaces of linear operators (see, for instance, \cite{DK} and the references therein), much less is known in nonlinear settings such as spaces of homogeneous polynomials and multilinear forms.

Our purpose is to provide a systematic study of the $w^*$-KK and the $w^*$-UKK for spaces of homogeneous polynomials and (symmetric) multilinear forms by presenting optimal results. Let $X$ be a real or complex Banach space. Our main focus is on the Banach spaces $\mathcal{P}(^n X; \K)$, $\mathcal{L}(^n X \times Y; \K)$ and $\mathcal{L}_s(^nX; \K)$ endowed with the canonical weak-star topology arising from their natural tensor preduals.

The main technical contribution of the paper is a family of concentration theorems for $2$-homogeneous polynomials and bilinear forms. Roughly speaking, these results show that whenever a 2-homogeneous polynomial or bilinear form almost attains its norm at finitely supported vectors, then it is uniformly close to its restriction to the corresponding finite set of coordinates. Thus, most of its behavior is already determined by finitely many coordinates. These concentration theorems are of independent interest and we expect that they might be useful in different contexts when one is studying the spaces of $2$-homogeneous polynomials and bilinear forms.

Having these concentration theorems in our favor, we are able to provide positive results about the $w^*$-UKK. Indeed, one can replace a $2$-homogeneous polynomial or bilinear form by a finite-dimensional restriction, where the weak-star and norm topologies coincide. This yields the $w^*$-UKK for spaces of complex $2$-homogeneous polynomials and complex bilinear forms on $p$-convex Banach sequence lattices with constant one whenever $p>2$. As applications, we obtain the $w^*$-UKK for the corresponding spaces over $c_0$, $\ell_p$, Lorentz sequence spaces $d(w,p)$ and Garling sequence spaces $g(w,p)$ whenever $p>2$. We also obtain several positive results in the real setting for spaces of operators and bilinear forms by combining concentration arguments with estimates involving moduli of convexity of power type. These methods in particular allow us to recover results from \cite{DK} and \cite{RZSSD}.

At this point, one might wonder why considering only $2$-homogeneous polynomials and bilinear forms. The answer comes with a collection of negative results showing that the previous positive theorems are essentially optimal. Indeed, we prove that the sequential $w^*$-KK fails for real homogeneous polynomials of every degree greater than one, for complex homogeneous polynomials of degree at least three and for (symmetric) multilinear forms of degree at least three. We also show that the restriction $p>2$ cannot be removed in the quadratic complex setting by proving that the corresponding properties fail on $\ell_p$, Lorentz spaces $d(w,p)$ and Garling spaces $g(w,p)$ whenever $1 < p \leq 2$ for the spaces of $2$-homogeneous polynomials on these spaces. 

The paper is organized as follows. In \hyperref[sec:preliminaries]{Section~\ref*{sec:preliminaries}}, we recall the necessary background on tensor products, homogeneous polynomials, multilinear mappings and Kadec-Klee properties. \hyperref[sec: concentration]{Section \ref*{sec: concentration}} contains the concentration theorems for $2$-homogeneous polynomials and bilinear forms, which constitute the main technical tool of the paper and, as we have mentioned before, are of independent interest. In \hyperref[sec: quadratic and bilinear]{Section \ref*{sec: quadratic and bilinear}}, we apply these concentration results to obtain the $w^*$-UKK for spaces of $2$-homogeneous polynomials and bilinear forms on a broad class of Banach sequence lattices, together with several applications. \hyperref[sec: higher degrees]{Section \ref*{sec: higher degrees}} is devoted to negative results, showing that the $w^*$-KK properties fail for higher degrees in these spaces. Finally, \hyperref[sec:further-remarks]{Section \ref*{sec:further-remarks}} contains further remarks on the topic. In particular, we discuss the optimality of the concentration theorems, explain how our arguments extend to the non-separable spaces $c_0(\Gamma)$ and $\ell_p(\Gamma)$, and derive consequences concerning the strong subdifferentiability of the corresponding tensor norms.

\section{Notation and preliminaries}\label{sec:preliminaries}

Throughout the paper, $X$ denotes a Banach space over the scalar field $\K=\R$ or $\C$. We write $B_X$ and $S_X$ for the closed unit ball and the unit sphere of $X$, respectively, and $X^*$ for the dual of $X$. Since here the distinction between the real and complex cases will be important, we write, for instance, $c_0(\R)$ and $c_0(\C)$, and similarly for the other spaces.

\subsection{Spaces we will be working with} Let $X_1,\ldots, X_n$ be Banach spaces over $\K$. We denote by $\mathcal{L}(^nX_1\times \cdots \times X_n;\K)$ the Banach space of all continuous $n$-linear forms endowed with the norm
\begin{equation*}
\|A\|=\sup\{|A(x_1,\ldots,x_n)|:x_j\in B_{X_j}\}
\end{equation*}
for every $A \in \mathcal{L}(^n X_1 \times \cdots \times X_n; \K)$. When all the spaces considered are the same $X$, we just write $\mathcal{L}(^nX;\K)$ for simplicity. This space is naturally identified, via the canonical linearization, with the dual of the corresponding projective tensor product. More precisely, we have 
\[
\mathcal{L}(^nX_1\times \cdots \times X_n;\K)=\bigl(X_1 \pten \cdots \pten X_n)^*
\]
isometrically. For the basic theory of tensor products and the notation used throughout the paper, we refer the reader to \cite{defant1992tensor} (see also \cite{ryan2002introduction}). 

We write $\mathcal{L}_s(^nX;\K)$ for the closed subspace of symmetric $n$-linear forms endowed with the same multilinear norm. If $\sigma^n$ is the symmetrization projection on $\multiten{n}X$, we denote its range, with the norm inherited from $\multiten{n}X$, by
\begin{equation*}
\symmultiten{n}X:=\sigma^n\bigl(\multiten{n}X\bigr).
\end{equation*}
Then, we have the identification
\begin{equation*}
\mathcal{L}_s(^nX;\K)=\bigl(\symmultiten{n}X\bigr)^*
\end{equation*}
isometrically. 

A continuous mapping $P:X\rightarrow\K$ is called $n$-homogeneous polynomial if there is a continuous symmetric $n$-linear form $\check P\in \mathcal{L}_s(^nX;\K)$ such that $P(x)=\check P(x,\ldots,x)$ for every $x\in X$. We denote the Banach space of all continuous scalar-valued $n$-homogeneous polynomials on $X$ by $\mathcal{P}(^nX;\K)$ and use the norm
\begin{equation*}
\|P\|=\sup\{|P(x)|:x\in B_X\}.
\end{equation*}
The canonical linearization identifies $\mathcal{P}(^nX;\K)$ isometrically with
\begin{equation*}
\bigl(\polten{n}X\bigr)^*
\end{equation*}
where $\polten{n}X$ denotes the completed symmetric projective tensor product associated with the polynomial norm (see, for instance, \cite{floret1997natural}). 

Let us notice that $\polten{n}X$ and $\symmultiten{n}X$ are the same set but generally carry different norms. Thus, although $\mathcal{P}(^nX;\K)$ and $\mathcal{L}_s(^nX;\K)$ are isomorphic, they need not be isometric. This is why weak-star Kadec-Klee properties for homogeneous polynomials and weak-star Kadec-Klee for symmetric multilinear forms should be viewed as different problems. 

There is another identification that will be used repeatedly throughout the paper. Let $Z$ be a Banach space. The space $\mathcal{L}(^{n+1}X_1\times\cdots \times X_n\times Z; \K)$ can be identified isometrically with the vector-valued $n$-linear mappings $\mathcal{L}(^{n}X_1\times\cdots \times X_n; Z^*)$ sending an $n+1$-linear form
$A$
to the vector-valued $n$-linear mapping
\(
\widetilde A
\)
defined by
\[
\widetilde A(x_1,\ldots,x_n)(z)
=A(x_1,\ldots,x_n,z)
\]
for every $z \in Z$. We shall use this identification repeatedly, especially in the bilinear case, where bilinear forms on $X\times Y$ can be regarded as bounded linear operators from $X$ into $Y^*$. This point of view will be particularly convenient, since some of the bibliography cited is originally formulated for linear operators.

We will also work with the following Banach spaces. Let us recall the definition of Lorentz spaces (see, for instance, \cite[Section 4.e]{LT1} and \cite[Section 2.a]{LT2}). Let $1 < p < \infty$. Let $(w(k))_{k=1}^{\infty}$ be a decreasing sequence of positive numbers such that $w(1) = 1$, $\lim_k w(k) = 0$ and $\sum_{k=1}^{\infty} w(k) = \infty$. Then, the corresponding Lorentz sequence space, denoted by $d(w,p)$, is defined as the space of all sequences $x=(x(k))_k$ such that
\begin{equation*}
    \|x\|_{d(w,p)} = \sup_{\pi \in \Sigma_{\N}} \left( \sum_{k=1}^{\infty} |x(\pi(k))|^p w(k) \right)^{1/p} = \left( \sum_{k=1}^{\infty} |x^{\star}(k)|^p w(k) \right)^{1/p} < \infty
\end{equation*}
where $\Sigma_{\N}$ denotes the group of permutations of the natural numbers and $(x^{\star}(k))_k$ denotes the decreasing rearrangement of $(|x(k)|)_k$. On the other hand, the Garling space $g(w,p)$ is defined by 
\begin{equation*}
    \|x\|_{g(w,p)}^p := \sup_{\phi} \sum_{j=1}^{\infty} |x(\phi(j))|^pw(j) 
\end{equation*}
where the supremum is taken over all increasing maps $\phi: \N \rightarrow \N$. Its canonical unit vectors form a normalized Schauder basis. The main difference between Lorentz sequence spaces and Garling sequence spaces is that, unlike for $d(w,p)$, the canonical basis of $g(w,p)$ is not symmetric, which will not play any role in this paper. For background, see, for instance, \cite{AAW, Garling} and the references therein.

\subsection{Kadec-Klee definitions} We introduce now the main properties we will be considering in the present manuscript. For this subsection, let us fix $X$ to be a Banach space over $\K$.

\begin{definition}\label{def:wstar-kk} We say that the dual space $X^*$ has the  weak-star Kadec-Klee property ($w^*$-KK for short) with respect to $X$ if, whenever $(x_\alpha^*)\subseteq S_{X^*}$ is a net, $x^*\in S_{X^*}$ and $x_\alpha^*\xrightarrow{w^*}x^*$, then $\|x_\alpha^*-x^*\|\rightarrow0$.
\end{definition}

In the literature, the sequential version of this property is also considered. We say that $X^*$ has the sequential $w^*$-KK property whenever \hyperref[def:wstar-kk]{Definition~\ref*{def:wstar-kk}} holds with sequences in place of nets.

In order to introduce the next definition, we first define the separation of a net. Given a net $(x_\alpha)$ in $X$, we write
\[
\sep(x_\alpha)=\inf\{\|x_\alpha-x_\beta\|:\alpha<\beta\}.
\]

We have the following stronger property.

\begin{definition}\label{def:wstar-ukk} The dual space $X^*$ has the weak-star uniform Kadec--Klee property ($w^*$-UKK, for short) if, for every $\eps>0$, there exists $\delta>0$ such that, whenever $(x_\alpha^*)_{\alpha\in\Lambda}\subseteq B_{X^*}$ is a net indexed by a directed set $\Lambda$ with no maximal element,  $x_\alpha^*\xrightarrow{w^*}x^*$and  $\sep(x_\alpha^*)\geq\eps$   then $\|x^*\|\leq1-\delta$. 
\end{definition}

Notice that excluding directed sets with a maximal element is harmless. Indeed, if a net $(x_\alpha)_{\alpha\in\Lambda}$ has a maximal index $\alpha_0$, then its limit (in any topology which is Hausdorff) is necessarily $x_{\alpha_0}$, and the net cannot provide any non-trivial information.

Analogously, we consider the sequential $w^*$-UKK by using sequences instead of nets. The sequential formulation is used in the classical $w^*$-UKK literature (see, for instance, \cite{DK,vanDulstSims}). The stronger net formulation of the $w^*$-UKK has also been studied, albeit in an equivalent form, by Lancien and Raja (see \cite{Lancien} and \cite{Raja}, respectively), where it is expressed in terms of weak-star neighborhoods of small diameter. Throughout the paper, we will be using this characterization of the \hyperref[def:wstar-ukk]{Definition \ref*{def:wstar-ukk}} as in \hyperref[obs:wstar-ukk]{Proposition \ref*{obs:wstar-ukk}} below without any explicit reference (see \cite[Lemma 3.1]{Raja} applied to the unit ball of the dual).

\begin{proposition}\label{obs:wstar-ukk} Let $X$ be a Banach space. Then, $X^*$ has the $w^*$-UKK if and only if for every $\eps > 0$, there exists $\eta > 0$ such that every $x^* \in B_{X^*}$ satisfying $\|x^*\| > 1 - \eta$ admits a weak-star neighborhood $U$ for which $\diam(U \cap B_{X^*})<\eps$.
\end{proposition}

There is a relevant comment we feel obliged to highlight. When we write, for instance, $\mathcal{P}(^n X; \R)$ we are considering the real-valued $n$-homogeneous polynomials viewed as a real Banach space in the same way that $\mathcal{P}(^n X; \C)$ means the complex-valued $n$-homogeneous polynomials considered as a complex Banach space. This distinction is relevant in our context as one could consider $\mathcal{P}_{\R}(^n X; \C)$ to be the Banach space of all complex-valued $n$-homogeneous polynomials as a real Banach space, for instance. We will not be working with this last space in the present manuscript.

The following lemma shows that, for polynomial and multilinear spaces, weak-star convergence is equivalent to pointwise convergence

\begin{lemma}\label{lem:wstar}
Let $X, X_1,\ldots, X_n$ be  Banach spaces over $\K$. The pointwise convergence for bounded nets on the Banach spaces $\mathcal{P}(^nX;\K),$ $\mathcal{L}(^nX_1\times \cdots \times X_n;\K)$ and $\mathcal{L}_s(^nX;\K)$ coincides with the weak-star convergence induced by the canonical preduals $\polten{n}X$, $\multiten{n}X_j$ and $\symmultiten{n}X$, respectively.
\end{lemma}

\begin{proof}
We first consider homogeneous polynomials. Let $(P_\alpha)\subseteq\mathcal{P}(^nX;\K)$ be a bounded net.  Since $P_\alpha(x)-P(x)=\langle P_\alpha-P,x^{\otimes n}\rangle$, the pointwise convergence is equivalent to the convergence of elementary tensors $x^{\otimes n}$. Since the linear span of the tensors $x^{\otimes n}$ is dense in $\polten{n}X$ and $(P_\alpha-P)$ is bounded in norm, convergence on this dense subspace extends to the whole predual. 

For multilinear forms, one uses the elementary tensors $x_1\otimes\cdots\otimes x_n$, whose linear span is dense in $\multiten{n}X$. On the other hand, the proof for symmetric multilinear forms is the same after applying the symmetrization projection and using the density of symmetrized elementary tensors in $\symmultiten{n}X$.
\end{proof}

Finally, we record the following result which relates the $w^*$-KK in $\mathcal{L}(^n X; \K)$ with the same property in $\mathcal{L}_s(^n X; \K)$.

\begin{fact}\label{prop:multi implies simetric}
Let $X$ be a Banach space over $\K$. Assume that $\mathcal{L}(^nX;\K)$ has the (sequential) $w^*$-KK, then $\mathcal{L}_s(^nX;\K)$ has the (sequential) $w^*$-KK. The same result is true for the (sequential) $w^*$-UKK.
\end{fact}

\begin{proof} This result is a consequence of the fact that given a net $(L_\alpha)\subset S_{\mathcal{L}_s(^nX;\K)}$ and  $L\in S_{\mathcal{L}_s(^nX;\K)}$, then $L_\alpha\xrightarrow{w^*}L$ if and only if the same convergence holds when they are regarded as elements of $\mathcal{L}(^nX;\K)$. Indeed, let
\begin{equation*}
\sigma^n: \multiten{n}X\rightarrow \symmultiten{n}X
\end{equation*}
be the symmetrization projection. Since $L_\alpha$ and $L$ are symmetric, we have that $L_\alpha(z)=L_\alpha(\sigma^n z)$ and $L(z)=L(\sigma^n z)$ for every $z\in\multiten{n}X$. Therefore, $L_\alpha(z)\rightarrow L(z)$  $\forall z\in\multiten{n}X$ if and only if $L_\alpha(w)\rightarrow L(w)$ $\forall w\in \symmultiten{n} X$. 
\end{proof}

\section{Concentration theorems}\label{sec: concentration}

One of the main contributions of this article is a family of concentration theorems for $2$-homogeneous polynomials and bilinear forms on $p$-convex Banach sequence lattices. Roughly speaking, these results show that if a polynomial or bilinear form almost attains its norm at vectors supported on finitely many coordinates, then most of the polynomial is already determined by those coordinates. In other words, the contribution coming from the remaining coordinates is small.

Beyond their role in the present work, these concentration theorems are of independent interest. They provide a localization principle for $2$-homogeneous polynomials and bilinear forms, showing that, under suitable geometric assumptions on the underlying Banach sequence lattice, the global behavior of these mappings is controlled by finite-dimensional pieces. We expect that this phenomenon may find further applications in the geometry of spaces of homogeneous polynomials and multilinear forms as well as in problems concerning norm-attainment and finite-dimensional approximation.

These concentration results also constitute the main technical ingredient for the weak-star Kadec--Klee properties established in the next section. Indeed, once concentration is available, the passage to finite-dimensional restrictions allows one to exploit the fact that weak-star and norm topologies coincide on finite-dimensional spaces.

We divide this section into two parts, devoted to the complex and real settings, respectively. Before treating them separately, we introduce the notation and basic definitions used throughout. In particular, we adopt the terminology of \hyperref[complex-Banach-sequence-lattice]{Definition \ref*{complex-Banach-sequence-lattice}} and \hyperref[def:r-convex]{Definition \ref*{def:r-convex}}, which provides a common framework for a broad range of applications of our results.

\begin{definition}\label{complex-Banach-sequence-lattice} A Banach sequence lattice is a linear subspace $X \subseteq \K^{\N}$ containing $c_{00}$ which is a Banach space and has the following property: whenever $x \in X$ and $y \in \K^{\N}$ satisfy $|y(k)| \leq |x(k)|$ for every $k \in \N$, then $y \in X$ and $\|y\|_X \leq \|x\|_X$.
\end{definition}

Let $X$ be a  Banach sequence lattice. For a finite set $F \subseteq \N$, the coordinate projection is denoted by 
\begin{equation*}
    \pi_F(x) := \sum_{k \in F} x(k) e_k 
\end{equation*}
for every $x \in X$. By the lattice property, we have that every $\pi_F$ is contractive for every finite set $F \subseteq \N$. In other words, $\|\pi_F(x)\|_X \leq \|x\|_X$ for every $x \in X$, that is, $\|\pi_F\| \leq 1$ for every $F \subseteq \N$ finite.

\begin{definition}\label{def:r-convex} Let $X$ be a Banach sequence lattice. For $1 \leq r < \infty$, the space $X$ is said to be $r$-convex with constant one if 
\begin{equation*} 
    \left\| \left( \sum_{j=1}^m |x^j|^r \right)^{1/r} \right\|_X \leq \left( \sum_{j=1}^m \|x^j\|_X^r \right)^{1/r}
\end{equation*}
for every finite family $x^1, \ldots, x^m \in X$. 
\end{definition}

We will use the fact that $p$-convexity with constant one implies $r$-convexity with constant one whenever $1 \leq r \leq p$. This follows from \cite[Proposition 1.d.5, page 49]{LT2}. In particular, a space which is $p$-convex with constant one for some $p>2$ is also 2-convex with constant one. It was proved in \cite{R} that $d(w,p)$ is $r$-convex with constant one for every $1 \leq r \leq p$. For the analogous result for the Garling spaces one can check \cite[Theorem 3.1.(b)]{AADK}. We will be using these results without any explicit reference. 

\subsection{The complex case} In this section, we assume throughout that $X$ is a complex Banach sequence lattice on $\N$. The following elementary lemma is well-known. Since we were unable to locate a direct reference and its proof is short, we include it for the reader's convenience. The constant $1/2$ is not optimal, but its precise value is irrelevant for our purposes, so we settle for this simpler estimate.

\begin{lemma} \label{lemma-sign-estimate} Let $z_1, \ldots, z_N \in \C$. There are signs $\sigma_1, \ldots, \sigma_N \in \{-1,1\}$ such that
\begin{equation*}
    \left| \sum_{k=1}^{N} \sigma_k z_k \right| \geq \frac{1}{2} \sum_{k=1}^N |z_k|.
\end{equation*}
\end{lemma}

\begin{proof} Since
$$\sum_{k=1}^N |z_k| \leq \sum_{k=1}^N |\im(z_k)| + \sum_{k=1}^N|\re(z_k)|$$ we may assume that
$$\frac{1}{2}\sum_{k=1}^N |z_k| \leq \sum_{k=1}^N  |\re(z_k)|$$
(the proof with the imaginary parts is completely analogous). Choose signs so that 
$\re(\sigma_k z_k)=|\re(z_k)|\ge 0$. Then
\begin{eqnarray*}
    \left| \sum_{k=1}^{N} \sigma_k z_k \right| &\geq& \left|  \re\left(\sum_{k=1}^{N} \sigma_k z_k \right)\right|=\left|\sum_{k=1}^{N} \re(\sigma_k z_k) \right| \\
    &=&\sum_{k=1}^{N} |\re(z_k)|  \geq  \frac{1}{2}\sum_{k=1}^N |z_k|.\
\end{eqnarray*}
\end{proof}

We begin with the $2$-homogeneous case. The following theorem shows that whenever a $2$-homogeneous polynomial almost attains its norm at a finitely supported vector, it is uniformly close to the polynomial obtained by restricting it to the support of that vector.

\begin{theorem} \label{lemma-p-convexity} Let $X$ be a complex Banach sequence lattice. Suppose that $X$ is $p$-convex with constant one for some $p > 2$. For every $\eps>0$, there exists $\eta > 0$ such that the following holds: whenever $F \subseteq \N$ is finite, $x_0 \in S_X$ is supported on $F$ and $Q \in B_{\mathcal{P}(^2X;\C)}$ satisfies $|Q(x_0)| > 1 - \eta$, then we have that 
\begin{equation*}
    \|Q - Q \circ \pi_F\| < \eps.
\end{equation*}
\end{theorem}

\begin{proof} Suppose that the conclusion is false. Then, there exist $\eps > 0$, finite sets $F_m \subseteq \N$, vectors $x_m \in S_X$ supported on $F_m$ and polynomials $Q_m \in B_{\mathcal{P}(^2 X; \C)}$ such that 
\begin{equation} \label{eq2}
|Q_m(x_m)| > 1 - \frac{1}{m} 
\end{equation}
and 
\begin{equation} \label{eq3}
\|Q_m - Q_m \circ \pi_{F_m}\| \geq \eps
\end{equation}
for every $m \in \N$. Notice that, by multiplying each $Q_m$ by a modulus-one scalar, we may assume that $Q_m(x_m) > 0$ without affecting the inequality (\ref{eq3}).

Let us give the general idea of the proof. Let $\check{Q}_m$ be the symmetric bilinear form associated with $Q_m$. If $z \in X$, then we can write
\begin{equation*}
    z = \pi_{F_m}(z) + (\id - \pi_{F_m})(z)
\end{equation*}
and the quadratic expansion gives 
\begin{equation} \label{final}
    Q_m(z) = Q_m(\pi_{F_m}z) + 2 \check{Q}_m(\pi_{F_m}z, (\id - \pi_{F_m})z) + Q_m((\id - \pi_{F_m})z)
\end{equation}
In this last expression, we have that $Q_m(\pi_{F_m}z)$ depends only on the coordinates in $F_m$ while $Q_m((\id - \pi_{F_m})z)$ depends only on the coordinates outside $F_m$. Notice also that $2 \check{Q}_m(\pi_{F_m}z, (\id - \pi_{F_m})z)$ combines one vector supported in $F_m$ and one vector supported outside $F_m$. In what follows, we will first control the part of $Q_m$ supported outside $F_m$ and then we will take care of the mixed part.

For a fixed $m$, fix $0 < t < 1$ and let $y \in B_X$ be supported on $\N \setminus F_m$. Since $x_m$ and $y$ have disjoint supports, for every $\lambda \in \T$, we have 
\begin{equation*}
    |(1 - t^p)^{1/p} x_m(k) + \lambda t y(k)|^p = (1-t^p)|x_m(k)|^p + t^p|y(k)|^p
\end{equation*}
for every $k \in \N$. Since $X$ is a Banach sequence lattice, we have, by using also the $p$-convexity with constant one, that
\begin{eqnarray*}
\|(1-t^p)^{1/p} x_m + \lambda t y \|_X &=& \| (|(1-t^p)^{1/p} x_m|^p + |\lambda t y |^p)^{1/p}\|_X \\
&\leq& \left( \|(1-t^p)^{1/p} x_m\|_X^p + \| \lambda t y \|_X^p \right)^{1/p} \\
&=& \left( (1-t^p) \|x_m\|_X^p + t^p \|y\|_X^p \right)^{1/p} \\
&\leq& \left( (1 - t^p) + t^p \right)^{1/p} \\
&=& 1
\end{eqnarray*}
for every $\lambda \in \T$. The same estimate holds with $\lambda$ replaced by $-\lambda$. Then, since $x_m, y \in B_X$ and $\|Q_m\| \leq 1$, we have that 
\begin{equation*}
    |Q_m((1 - t^p)^{1/p} x_m + \lambda t y)| \leq 1 \ \ \ \mbox{and} \ \ \ |Q_m((1-t^p)^{1/p} x_m - \lambda t y)| \leq 1.
\end{equation*}
Therefore,
\begin{equation*}
    \left| \frac{Q_m((1-t^p)^{1/p} x_m + \lambda t y) + Q_m((1-t^p)^{1/p} x_m - \lambda t y)}{2} \right| \leq 1.
\end{equation*}
which implies, using the binomial expansion, that 
\begin{equation*}
    |(1 - t^p)^{2/p} Q_m(x_m) + \lambda^2 t^2 Q_m(y)| \leq 1
\end{equation*}
for every $\lambda \in \T$. Now, choose $\lambda \in \T$ so that $\lambda^2 Q_m(y) = |Q_m(y)|$. Since  $Q_m(x_m)>0$,
\begin{equation*}
    (1 - t^p)^{2/p} Q_m(x_m) + t^2 |Q_m(y)| \leq 1 
\end{equation*}
which implies that 
\begin{equation*}
    |Q_m(y)| \leq \frac{1 - (1-t^p)^{2/p}Q_m(x_m)}{t^2}
\end{equation*}
for every $0 < t < 1$, for every $m \in \N$ and for every $y \in B_X$ with $\supp y \subseteq \N \setminus F_m$. We can take the supremum over all such $y$ to get 
\begin{equation*}
    \sup_{\substack{\|y\|_X\leq 1\\
\operatorname{supp}(y)\cap F_m=\emptyset}} |Q_m(y)| \leq \frac{1 - (1-t^p)^{2/p} Q_m(x_m)}{t^2} \stackrel{(\ref{eq2})}{<} \frac{1 - (1-t^p)^{2/p}\left(1 - \frac{1}{m}\right)}{t^2}
\end{equation*}
for every $m \in \N$. Now, we take the upper limit in $m$ and obtain
\begin{equation*} 
\limsup_{m \rightarrow \infty} \sup_{\substack{\|y\|_X\leq 1\\
\operatorname{supp}(y)\cap F_m=\emptyset}} |Q_m(y)| \leq \frac{1 - (1-t^p)^{2/p}}{t^2}
\end{equation*}
for every $0 < t < 1$ fixed. Since $p > 2$, we have that 
\begin{equation*}
    \lim_{t \rightarrow 0^+} \frac{1 - (1-t^p)^{2/p}}{t^2} = 0
\end{equation*}
(notice that if $p=2$, this limit would be 1!). Therefore, we have
\begin{equation} \label{eq4}
\limsup_{m \rightarrow \infty} \sup_{\substack{\|y\|_X\leq 1\\
\supp (y)\cap F_m=\emptyset}} |Q_m(y)| = 0.
\end{equation}

We now show that the mixed part converges uniformly to zero. Otherwise, after passing to a subsequence and decreasing $\eps$ if necessary, there are vectors $a_m, y_m \in B_X$ such that $\supp a_m \subseteq F_m$, $\supp y_m \cap F_m = \emptyset$ and $|2 \check{Q}_m(a_m, y_m)| \geq \eps$ for every $m \in \N$. Fix $m \in \N$. So, $y_m$ is also fixed. Consider the mapping $a \mapsto 2 \check{Q}_m(a,y_m)$ defined on vectors $a$ supported on $F_m$. Since $\check{Q}_m$ is bilinear and $y_m$ is fixed, this map is linear in $a$. Now, if $a$ is supported on a finite set $F_m$, then we can write

\begin{equation*}
    a = \sum_{k \in F_m} a(k) e_k. 
\end{equation*}
Using linearity in the first variable, we obtain
\begin{eqnarray*}
    2 \check{Q}_m(a,y_m) = 2 \check{Q}_m\left( \sum_{k \in F_m} a(k) e_k, y_m \right) &=& \sum_{k \in F_m} a(k) \cdot 2 \check{Q}_m(e_k, y_m).
\end{eqnarray*}
This means we can write 
\begin{equation*}
    2 \check{Q}_m(a, y_m) = \sum_{k \in F_m} c_{m,k} a(k)
\end{equation*}
for every vector $a$ supported on $F_m$ for some scalar $c_{m,k}$.

 For each $k \in F_m$, with fixed $m$, multiply each coordinate of $a_m$ by a scalar of modulus one such that for the resulting vector $\tilde{a}_m$, we have 
 \begin{equation}\label{eq: a ort to x}
     \re (\overline{x_m(k)} \tilde{a}_m(k)) = 0.
 \end{equation}Moreover, by \hyperref[lemma-sign-estimate]{Lemma \ref*{lemma-sign-estimate}}, eventually changing the signs of $\tilde{a}_m(k)$ we may also assume that 
\begin{eqnarray*}
|2 \check{Q}_m(\tilde{a}_m, y_m)|  &=& 
\left|\sum_{k \in F_m} c_{m,k}\tilde{a}_m(k)\right| 
\geq \frac{1}{2}\sum_{k \in F_m} \left|c_{m,k}\tilde{a}_m(k)\right|\\
&=&\frac{1}{2}\sum_{k \in F_m} \left|c_{m,k}a_m(k)\right| 
\ge \frac{1}{2}|2\check{Q}_m(a_m, y_m)|\geq \frac{\varepsilon}{2}
\end{eqnarray*}
Since $|\tilde{a}_m(k)|=|a_m(k)|$, for every $k$, we have that $\|\tilde{a}_m\|_X = \|a_m\|$. By \eqref{eq: a ort to x}, for any $s > 0$,
\begin{eqnarray*}
    |x_m(k) + s \tilde{a}_m(k)|^2 &=& |x_m(k)|^2 + s^2 |\tilde{a}_m(k)|^2 + 2s \re (\overline{x_m(k)} \tilde{a}_m(k)) \\
    &=& |x_m(k)|^2 + s^2 |\tilde{a}_m(k)|^2 
\end{eqnarray*}
for every $k \in \N$. This implies that 
\begin{equation*}
    |x_m(k) + s \tilde{a}_m(k)| = \left(|x_m(k)|^2 + s^2 |\tilde{a}_m(k)|^2 \right)^{1/2}
\end{equation*}
for every $k \in \N$ and every $s > 0$. As $p$-convexity implies 2-convexity with constant one, we have that 
\begin{equation*}
    \|x_m + s \tilde{a}_m\|_X \leq (1 + s^2)^{1/2}
\end{equation*}
for every $s > 0$. The vector $y_m$ is disjointly supported from both $x_m$ and $\tilde{a}_m$. Hence,
\begin{equation*}
    \|x_m + s \tilde{a}_m + \mu t y_m \|_X^2 \leq ((1 + s^2)^{p/2} + t^p)^{2/p}
\end{equation*}
for every $s, t > 0$ and for all $\mu \in \T$. The same estimate holds true for $x_m - s \tilde{a}_m - \mu t y_m$. As in the first part of the proof, we have that 
\begin{equation*}
    \left| \frac{Q_m(x_m + s \tilde{a}_m + \mu t y_m) + Q_m(x_m - s \tilde{a}_m - \mu t y_m)}{2} \right| \leq ((1 + s^2)^{p/2} + t^p)^{2/p}.
\end{equation*}
By the binomial expansion, the expression 
\begin{equation*}
    \frac{Q_m(x_m + s \tilde{a}_m + \mu t y_m) + Q_m(x_m - s \tilde{a}_m - \mu t y_m)}{2}
\end{equation*}
is equal to
\begin{equation*}
    Q_m(x_m) + s^2 Q_m(\tilde{a}_m) + \mu s t 2 \check{Q}_m(\tilde{a}_m, y_m) + \mu^2 t^2 Q_m(y_m).
\end{equation*}
Choose $\mu \in \T$ so that $\mu \cdot 2 \check{Q}_m(\tilde{a}_m, y_m) = |2 \check{Q}_m(\tilde{a}_m, y_m)|$. Therefore, 
\begin{equation*}
    |Q_m(x_m) + s^2 Q_m(\tilde{a}_m) + st |2 \check{Q}_m (\tilde{a}_m, y_m)| + \mu^2 t^2 Q_m(y_m)| \leq ((1+s^2)^{p/2} + t^p)^{2/p}.
\end{equation*}
Taking real parts, we obtain
\begin{equation*}
    Q_m(x_m) + s^2 \re Q_m(\tilde{a}_m) + st | 2 \check{Q}_m(\tilde{a}_m, y_m)| + t^2 \re (\mu^2 Q_m(y_m)) \leq ((1 + s^2)^{p/2} + t^p)^{2/p}.
\end{equation*}
Therefore, we get that
\begin{align*}
st \bigl|2\check{Q}_m(\widetilde{a}_m,y_m)\bigr| &\leq
\bigl((1+s^2)^{p/2}+t^p\bigr)^{2/p} -Q_m(x_m) -s^2\re Q_m(\widetilde{a}_m) \\ &\qquad -t^2\re\bigl(\mu^2Q_m(y_m)\bigr) \\
&\leq \bigl((1+s^2)^{p/2}+t^p\bigr)^{2/p}
-Q_m(x_m) +s^2 +t^2|Q_m(y_m)|.
\end{align*}
Now, notice that 
\begin{eqnarray*}
((1 + s^2)^{p/2} + t^p)^{2/p} &=& (1 + s^2) \left( 1 + \frac{t^p}{(1+s^2)^{p/2}} \right)^{2/p} \\
&\leq& (1 + s^2) \left( 1 + \frac{2}{p} \cdot \frac{t^p}{(1+s^2)^{p/2}} \right) \\
&=& 1 + s^2 + \frac{2}{p} \cdot t^p \cdot (1+s^2)^{1-\frac{p}{2}} \\
&\leq& 1 + s^2 + \frac{2}{p} \cdot t^p.
\end{eqnarray*}
Here, we have used that since $2/p < 1$, $(1+u)^{2/p} \leq 1 + \frac{2}{p}u$ for every $u \geq 0$ and also that $1-\frac{p}{2} = \frac{2-p}{2} < 0$. This implies that 
\begin{eqnarray*}
    st | 2 \check{Q}_m(\tilde{a}_m, y_m)| &\leq& 1 + s^2 + \frac{2}{p} t^p - Q_m(x_m) + s^2 + t^2|Q_m(y_m)| \\
    &=& 1 - Q_m(x_m) + 2s^2 + \frac{2}{p} t^p + t^2 |Q_m(y_m)|.
\end{eqnarray*}
Now, we calculate the $\limsup$ as $m \rightarrow \infty$ keeping $s$ and $t$ fixed. We know that $Q_m(x_m) \rightarrow 1$ and by $(\ref{eq4})$, $|Q_m(y_m)| \rightarrow 0$, so 
\begin{equation*}
    \frac{\eps}{2} st \leq \limsup_{m \rightarrow \infty} st |2 \check{Q}_m(\tilde{a}_m, y_m)| \leq 2s^2 + \frac{2}{p} t^p.
\end{equation*}
Therefore, the inequality 
\begin{equation*}
    \frac{\eps}{2} st \leq 2s^2 + \frac{2}{p} t^p 
\end{equation*}
holds for every $s,t > 0$. Take now $s = t^{p/2}$. So, $s^2 = t^p$ and 
\begin{equation*}
    \frac{\eps}{2} t^{\frac{p}{2}+1} \leq \left(2 + \frac{2}{p} \right) t^p 
\end{equation*}
which implies that 
\begin{equation*}
    \frac{\eps}{2} \leq \left(2 + \frac{2}{p} \right)t^{p - \left(\frac{p}{2}+1 \right)} = \left(2 + \frac{2}{p} \right)t^{\frac{p}{2}-1}.
\end{equation*}
Since $p > 2$, $\frac{p}{2}-1 > 0$ and $t^{\frac{p}{2}-1} \rightarrow 0$ as $t \rightarrow 0^+$. Therefore, $\frac{\eps}{2} \leq 0$, which is a contradiction. Therefore, 
\begin{equation} \label{eq7}
\limsup_{m \rightarrow \infty} \sup_{\substack{\|a\|_X \leq 1, \supp a \subseteq F_m\\
\|y\|_X \leq 1, \supp y \cap F_m = \emptyset}} |2 \check{Q}_m(a,y)| = 0.
\end{equation}

Finally, for every $z \in B_X$, we can write (\ref{final}), that is,
\begin{equation*}
    Q_m(z) - Q_m(\pi_{F_m}z) = 2 \check{Q}_m(\pi_{F_m}z, (\id - \pi_{F_m})z) + Q_m((\id - \pi_{F_m})z)
\end{equation*}
Since $\|\pi_{F_m} z\|_X \leq 1$ and $\|(\id - \pi_{F_m})z\|_X \leq 1$, we obtain
\begin{align*}
\limsup_{m\to\infty} \|Q_m-Q_m\circ\pi_{F_m}\| &\leq
\limsup_{m\to\infty} \sup_{\substack{\|y\|_X\leq1\\
    \supp(y)\cap F_m=\emptyset
}} |Q_m(y)| \\ &\quad + \limsup_{m\to\infty} \sup_{\substack{     \|a\|_X\leq1,\ \supp(a)\subseteq F_m\\ \|y\|_X\leq1,\ \supp(y)\cap F_m=\emptyset }} \bigl|2\check{Q}_m(a,y)\bigr|.
\end{align*}
By \eqref{eq4} and \eqref{eq7}, both terms on the right-hand side are equal to $0$. Hence,
\[
\limsup_{m\to\infty}\|Q_m-Q_m\circ\pi_{F_m}\|=0,
\]
which contradicts \eqref{eq3}. This completes the proof.
\end{proof}

We next prove the corresponding concentration result for bilinear forms. Although the argument follows the same general strategy as in the polynomial case, the bilinear setting requires controlling separately the two variables, which leads to a decomposition into three terms instead of two.

We will use the following notation. For finite sets $F, G \subseteq \N$, let us write 
\begin{equation} \label{notation-R-F-G}
R_{F,G} A(x,y) := A(\pi_F x, \pi_G y) 
\end{equation}
for every $(x, y) \in X \times Y$ and for every $A \in \mathcal{L}(^2 X \times Y, \K)$.

\begin{theorem}\label{thm: conc-bilin} Let $X, Y$ be complex Banach sequence lattices. Suppose that $X$ and $Y$ are $p$-convex with constant one for some $p > 2$. For every $\eps > 0$, there exists $\eta > 0$ such that the following holds: whenever $F, G \subseteq \N$ are finite, $x_0\in S_X$ and $y_0\in S_Y$ are supported on $F$ and $G$, respectively, and $A \in B_{\mathcal{L}(^2X\times Y; \C)}$ satisfies $|A(x_0, y_0)| > 1 - \eta$, then 
\begin{equation*}
    \|A - R_{F,G} \circ A\| < \eps.
\end{equation*}
\end{theorem}

\begin{proof}

Assume, by contradiction, that the conclusion does not hold. Then, there exist
$\eps>0$, finite sets $F_m,G_m\subseteq\mathbb N$, vectors
$x_m\in S_X$ and $y_m\in S_Y$ supported on $F_m$ and $G_m$, respectively, and bilinear forms
\(
A_m\in B_{\mathcal{L}(^2X\times Y;\mathbb C)}
\)
such that
\begin{equation}\label{eq:bil1}
|A_m(x_m,y_m)|>1-\frac1m
\end{equation}
and
\begin{equation}\label{eq:bil2}
\|A_m-R_{F_m,G_m}\circ A_m\|\geq \eps
\end{equation}
for every $m\in\mathbb N$.

Multiplying each $A_m$ by a scalar of modulus one, we may assume that
\[
A_m(x_m,y_m)>0
\]
without affecting \eqref{eq:bil2}.

For every $z\in X$, $w\in Y$, we have
\[
z=\pi_{F_m}z+(I-\pi_{F_m})z
\]
and
\[
w=\pi_{G_m}w+(I-\pi_{G_m})w.
\]
Hence, by bilinearity,
\begin{align}\label{eq:bil-decomp}
&A_m(z,w)-A_m(\pi_{F_m}z,\pi_{G_m}w) \\
={}& A_m(\pi_{F_m}z,(I-\pi_{G_m})w)
+A_m((I-\pi_{F_m})z,\pi_{G_m}w) \nonumber\\
&+A_m((I-\pi_{F_m})z,(I-\pi_{G_m})w). \nonumber
\end{align}
Therefore, it is enough to prove that each of the three terms in the right-hand
side converges uniformly to zero.

We first consider the term supported outside both $F_m$ and $G_m$. Let
$u\in B_X,v\in B_Y$ satisfy
\[
\supp u\cap F_m=\emptyset,
\qquad
\supp v\cap G_m=\emptyset .
\]
Fix $0<t<1$ and $\lambda,\mu\in\mathbb T$. Since the supports of $x_m$ and $u$
are disjoint, and the supports of $y_m$ and $v$ are disjoint, the
$p$-convexity gives
\[
\|(1-t^p)^{1/p}x_m+\lambda tu\|_X\leq1
\]
and
\[
\|(1-t^p)^{1/p}y_m+\mu tv\|_Y\leq1 .
\]
Consequently,
\[
\left|
A_m((1-t^p)^{1/p}x_m+\lambda tu,
(1-t^p)^{1/p}y_m+\mu tv)
\right|\leq1 .
\]
The same estimate holds after replacing $(\lambda,\mu)$ by
$(-\lambda,-\mu)$. Taking the average of the two expressions and using
bilinearity, we obtain
\[
\left|
(1-t^p)^{2/p}A_m(x_m,y_m)
+t^2\lambda\mu A_m(u,v)
\right|\leq1 .
\]
Choosing $\lambda,\mu\in\mathbb T$ such that
\[
\lambda\mu A_m(u,v)=|A_m(u,v)|,
\]
we get
\[
|A_m(u,v)|
\leq
\frac{1-(1-t^p)^{2/p}A_m(x_m,y_m)}{t^2}.
\]
Taking the supremum over all such $u$ and $v$, and using \eqref{eq:bil1}, yields
\[
\sup_{\substack{\|u\|\leq1,\ \supp u\cap F_m=\emptyset\\
\|v\|\leq1,\ \supp v\cap G_m=\emptyset}}
|A_m(u,v)|
\leq
\frac{1-(1-t^p)^{2/p}(1-\frac1m)}{t^2}.
\]
Therefore,
\[
\limsup_{m\to\infty}
\sup_{\substack{\|u\|\leq1,\ \supp u\cap F_m=\emptyset\\
\|v\|\leq1,\ \supp v\cap G_m=\emptyset}}
|A_m(u,v)|
\leq
\frac{1-(1-t^p)^{2/p}}{t^2}.
\]
Since $p>2$,
\[
\lim_{t\to0^+}\frac{1-(1-t^p)^{2/p}}{t^2}=0,
\]
and hence
\begin{equation*}
\limsup_{m\to\infty}
\sup_{\substack{\|u\|\leq1,\ \supp u\cap F_m=\emptyset\\
\|v\|\leq1,\ \supp v\cap G_m=\emptyset}}
|A_m(u,v)|=0 .
\end{equation*}

We now show that the first mixed term converges uniformly to zero. Suppose,
otherwise, that after passing to a subsequence there exist vectors
$a_m\in B_X, v_m\in B_Y$ such that
\[
\supp a_m\subseteq F_m,\qquad
\supp v_m\cap G_m=\emptyset,
\]
and
\begin{equation*}
|A_m(a_m,v_m)|\geq \varepsilon
\end{equation*}
for every $m\in\mathbb N$.

As in the proof of the polynomial case, by a sign-selection argument we can
choose $\widetilde a_m$ such that
\[
|\widetilde a_m(k)|=|a_m(k)|
\]
and
\[
\operatorname{Re}(\overline{x_m(k)}\widetilde a_m(k))=0
\]
for every $k\in\mathbb N$, and such that
\begin{equation}\label{eq:sign}
|A_m(\widetilde a_m,v_m)|
\geq \frac{1}{2}|A_m(a_m,v_m)|
\geq \frac{1}{2}\varepsilon .
\end{equation}

The first property implies that
\[
\|\widetilde a_m\|_X=\|a_m\|_X\leq1.
\]
Moreover, since $p$-convexity with constant one implies $2$-convexity with
constant one, for every $s>0$ we have
\[
\|x_m+s\widetilde a_m\|_X\leq (1+s^2)^{1/2}.
\]

Fix $s,t>0$ and $\mu\in\mathbb T$. Since $y_m$ and $v_m$ have disjoint
supports, the $p$-convexity of $Y$ yields
\[
\|(1-t^p)^{1/p}y_m+\mu tv_m\|_Y\leq1 .
\]
Therefore,
\[
\left|
A_m(x_m+s\widetilde a_m,
(1-t^p)^{1/p}y_m+\mu tv_m)
\right|
\leq
(1+s^2)^{1/2}.
\]
The same estimate holds for
\[
A_m(x_m-s\widetilde a_m,
(1-t^p)^{1/p}y_m-\mu tv_m).
\]
Hence, taking the average of these two expressions, we obtain
\[
\begin{aligned}
\Bigg|
&\frac{
A_m(x_m+s\widetilde a_m,
(1-t^p)^{1/p}y_m+\mu tv_m)
}{2}
\\
&+
\frac{
A_m(x_m-s\widetilde a_m,
(1-t^p)^{1/p}y_m-\mu tv_m)
}{2}
\Bigg|
\leq (1+s^2)^{1/2}.
\end{aligned}
\]

Using bilinearity, the expression inside the absolute value becomes
\[
(1-t^p)^{1/p}A_m(x_m,y_m)
+s\mu t A_m(\widetilde a_m,v_m).
\]

Choosing $\mu\in\mathbb T$ such that
\[
\mu A_m(\widetilde a_m,v_m)
=
|A_m(\widetilde a_m,v_m)|,
\]
 we get
\[
(1-t^p)^{1/p}A_m(x_m,y_m)
+s t |A_m(\widetilde a_m,v_m)|
\leq
(1+s^2)^{1/2}.
\]
Consequently,
\[
st |A_m(\widetilde a_m,v_m)|
\leq
(1+s^2)^{1/2}
-(1-t^p)^{1/p}A_m(x_m,y_m).
\]

Using the estimates
\[
(1+s^2)^{1/2}\leq 1+s^2
\]
and, for every $0 < t < 1$,
\[
(1 - t^p)^{1/p} \geq 1 - t^p,
\]
we obtain
\begin{eqnarray*}
st |A_m(\tilde{a}_m, v_m)| &\leq& (1 + s^2)^{1/2} - (1-t^p)^{1/p} A_m(x_m, y_m) \\
&\leq& 1 + s^2- (1-t^p)A_m(x_m, y_m) \\
&\leq& 1 - A_m(x_m, y_m) + s^2 + t^p.
\end{eqnarray*}
Taking the upper limit as $m\to\infty$, and using
\(
A_m(x_m,y_m)\longrightarrow1
\)
together with \eqref{eq:sign}, we obtain
\begin{equation*}
\frac{\eps}{2} st \leq s^2 + t^p
\end{equation*}
Taking $s=t^{p/2}$ gives
\begin{equation*}
\frac{\eps}{2} \leq 2 t^{\frac{p}{2}-1}
\end{equation*}
Since $p>2$, letting $t\to0^+$ yields a contradiction. Hence,
\[
\limsup_{m\to\infty}
\sup_{\substack{\|a\|\leq1,\ \supp a\subseteq F_m\\
\|v\|\leq1,\ \supp v\cap G_m=\emptyset}}
|A_m(a,v)|=0 .
\]

The second mixed term is obtained by the same argument, interchanging the two
variables. Therefore,
\begin{equation*}
\limsup_{m\to\infty}
\sup_{\substack{\|u\|\leq1,\ \supp u\cap F_m=\emptyset\\
\|b\|\leq1,\ \supp b\subseteq G_m}}
|A_m(u,b)|=0 .
\end{equation*}

Finally, taking the supremum in \eqref{eq:bil-decomp} over $z\in B_X$, $w\in B_Y$, and
using the three estimates above, we obtain
\[
\limsup_{m\to\infty}
\|A_m-R_{F_m,G_m}\circ A_m\|=0,
\]
which contradicts \eqref{eq:bil2}. Therefore, the result follows.
\end{proof}

\subsection{The real case} In this subsection, we focus on concentration theorems in the real setting. Although all the results remain valid over the complex field, we shall work throughout with real Banach spaces, since all the applications developed here concern the real case.

Given a Banach space $X$, we define its modulus of convexity, denoted  $\delta_X(\eps)$, by 
\begin{equation*}
    \delta_X(\eps) := \inf\left\{ 1 - \frac{\|x+y\|}{2}: x,y \in B_X \ \mbox{and} \ \|x - y\| \geq \eps \right\} 
\end{equation*}
for every $0 < \eps \leq 2$. We say that $X$ is uniformly convex if $\delta_X(\eps)> 0$ for every $\eps > 0$ (see, for instance, \cite[Definition 9.2]{FHHMZ}).

We say that a Banach space $X$ has modulus of convexity of power type $q$ for $q \geq 2$ if there exists $C > 0$ such that $\delta_X(\eps) \geq C \eps^q$ for every $0 < \eps \leq 2$. We have that $L_q$-spaces with $q \geq 2$ are examples of Banach spaces which have modulus of convexity of power type $q$. Also, if $1 < p < 2$, $L_p$ has modulus of convexity of power type 2 (see \cite[page 63]{LT2}).

The following result extends \cite[Theorem 3.1]{RZSSD} from $\ell_p$-spaces to $p$-convex Banach sequence lattices with constant one. Although our presentation has been formulated in terms of multilinear forms, we state the next result for linear operators in order to facilitate comparison with its original formulation in \cite[Theorem 3.1]{RZSSD}. Since the proof follows the same underlying argument, we omit it.

\begin{proposition} \label{real-fact-1} Let $X$ be a Banach sequence lattice. Suppose that $X$ is $p$-convex with constant one. Let $Z$ be a Banach space whose modulus of convexity has power type $2\leq q < p$. For every $\eps > 0$, there exists $\eta > 0$ with the following property: whenever $F \subseteq \N$ is finite, $x_0 \in S_X$ is supported on $F$ and $T \in B_{\mathcal{L}(X, Z)}$ satisfies $\|Tx_0\| > 1 - \eta$, we have that 
\begin{equation*}
    \|T( \id - \pi_F)\| < \eps.
\end{equation*}
\end{proposition}

As pointed out in \cite[Lemma 3.6]{RZSSD} (see also \cite[Lemma 2.11]{DantasKimLeeMazzitelli}), a concentration lemma for $c_0$ is much simpler than \hyperref[real-fact-1]{Proposition~\ref*{real-fact-1}}. Indeed, we have the following result. We also omit its proof.

\begin{proposition} \label{real-fact-2} Let $Z$ be a uniformly convex Banach space. For every $\eps > 0$, there exists $\eta > 0$ such that, whenever $F \subseteq \N$ is finite, $x_0 \in S_{c_0}$ is supported on $F$ and $T \in B_{\mathcal{L}(c_0, Z)}$ satisfies $\|Tx_0\| > 1 - \eta$, we have that $\|T(\id - \pi_F)\| < \eps$.
\end{proposition}

We would like now to use \hyperref[real-fact-1]{Proposition~\ref*{real-fact-1}} to get the following concentration theorem. We use the same definition as in (\ref{notation-R-F-G}) above.

We now show how \hyperref[real-fact-1]{Proposition~\ref*{real-fact-1}} can be used to obtain a concentration result for bilinear forms. The argument simply applies the operator-valued concentration theorem to each variable in turn.

\begin{theorem} \label{real-fact-3} Let $X$ and $Y$ be Banach sequence lattices. Suppose the following.
\begin{itemize}
\item[1.] $X$ is $p$-convex with constant one.
\item[2.] $Y$ is $s$-convex with constant one.
\item[3.] $Y^*$ has modulus of convexity of power type $2\leq q < p$.
\item[4.] $X^*$ has modulus of convexity of power type $2\leq r < s$.
\end{itemize}
Then, for every $\eps > 0$, there exists $\eta > 0$ with the following property. Whenever $x_0 \in S_X$, $y_0 \in S_Y$ are finitely supported with $F = \supp x_0$ and $G = \supp y_0$, and $A \in B_{\mathcal{L}(^2 X \times Y; \K)}$ satisfies $|A(x_0, y_0)| > 1 - \eta$, then 
\begin{equation*}
    \|A - R_{F,G} \circ A \| < \eps.
\end{equation*}
\end{theorem}

\begin{proof} Apply \hyperref[real-fact-1]{Proposition~\ref*{real-fact-1}} with $\eps/2$ to operators from $X$ into $Y^*$. Then, apply \hyperref[real-fact-1]{Proposition~\ref*{real-fact-1}} once again to operators from $Y$ into $X^*$ with $\eps/2$. Now take $\eta > 0$ to be the minimum between them.

Let $x_0 \in S_X$ and $y_0 \in S_Y$ have finite supports $F$ and $G$, respectively. Let $A \in B_{\mathcal{L}(^2 X \times Y; \K)}$ satisfy $|A(x_0, y_0)| > 1 - \eta$. Viewing $A$ as an operator $x \mapsto A(x, \cdot)$ from $X$ into $Y^*$, we have that 
\begin{equation*}
    \|A(x_0, \cdot)\| = \sup_{y \in B_Y} |A(x_0, y)| \geq |A(x_0, y_0)| > 1 - \eta 
\end{equation*}
and \hyperref[real-fact-1]{Proposition~\ref*{real-fact-1}} gives 
\begin{equation*}
    \sup_{x \in B_X} \|A((\id - \pi_F)x, \cdot) \| < \frac{\eps}{2} 
\end{equation*}
which is equivalent to say that 
\begin{equation*}
    |A((\id - \pi_F)x, y)| < \frac{\eps}{2} 
\end{equation*}
for every $x \in B_X$ and every $y \in B_Y$. Now, for every $x \in B_X$ and $y \in B_Y$, we can write $x = \pi_Fx + (\id - \pi_F)x$ and $y = \pi_Gy + (\id - \pi_G)y$ and, therefore, by bilinearity, we obtain
\begin{eqnarray*}
|A(x,y) - A(\pi_F x, \pi_G y)| &\leq& |A(\pi_Fx, (\id - \pi_G)y)| + |A((\id - \pi_F)x, y)| \\
&<& \frac{\eps}{2} + \frac{\eps}{2} = \eps.
\end{eqnarray*}
Taking the supremum over $x \in B_X$ and $y \in B_Y$, we are done.
\end{proof}

\section{The $w^*$-UKK for polynomials and bilinear forms}\label{sec: quadratic and bilinear}

This section contains the main applications of the concentration results established in the previous section. The key idea is that if a norm-one $2$-homogeneous polynomial or bilinear form almost attains its norm at finitely supported vectors, then it is uniformly close to its finite-dimensional restriction. Since finite-dimensional ranges identify weak-star and norm topologies, this concentration phenomenon naturally leads to the $w^*$-UKK.

\subsection{Applications in the complex case} We first prove the $w^*$-UKK property for spaces of $2$-homogeneous polynomials and bilinear forms on $p$-convex Banach sequence lattices. We then show that these abstract results apply to several classical families of sequence spaces, including $c_0$, $\ell_p$, Lorentz and Garling sequence spaces, as well as to suitable direct sums. Finally, we present examples showing that the condition $p>2$ is optimal in the settings considered here (see \hyperref[cor:lp-lorentz-garling-classification]{Corollary \ref*{cor:lp-lorentz-garling-classification}}).

\begin{theorem} \label{theorem-p-convexity} Let $X$ be a complex Banach sequence lattice such that $c_{00}$ is dense in $X$. If $X$ is $p$-convex with constant one for some $p > 2$, then $\mathcal{P}(^2 X; \C)$ has the $w^*$-UKK.
\end{theorem}

\begin{proof} Fix $\eps > 0$. Let us apply \hyperref[lemma-p-convexity]{Lemma \ref*{lemma-p-convexity}} with $\eps/4$ and let $\eta > 0$ be the resulting constant we get from that lemma. Let $Q \in B_{\mathcal{P}(^2 X; \C)}$ satisfy $\|Q\| > 1 - \frac{\eta}{4}$. Since finitely supported vectors are dense in $X$, there exists a finitely supported vector $x_0 \in S_X$ such that $|Q(x_0)| > 1 - \frac{\eta}{2}$. Let $F:= \supp x_0$. Consider the $w^*$-neighborhood of $Q$ given by 
\begin{equation*}
    U_0:= \left\{ R \in \mathcal{P}(^2X; \C): |R(x_0) - Q(x_0)| < \frac{\eta}{2} \right\}.
\end{equation*}
If $R \in U_0 \cap B_{\mathcal{P}(^2 X; \C)}$, then $|R(x_0)| > 1 -\eta$ and \hyperref[lemma-p-convexity]{Lemma \ref*{lemma-p-convexity}} yields $\|R - R \circ \pi_F\| < \eps/4$. The map $R \mapsto R \circ \pi_F$ is $w^*$-continuous and has finite-dimensional range. Then, there exists a $w^*$-neighborhood $U_1$ of $Q$ such that $\|R \circ \pi_F - Q \circ \pi_F\| < \eps/4$ for every $R \in U_1$. Put $U:= U_0 \cap U_1$. If $R, S \in U \cap B_{\mathcal{P}(^2 X; \C)}$, then 
\begin{align*}
\|R - S\| \leq \|R - R \circ \pi_F\| &+ \|R \circ \pi_F - Q \circ \pi_F\| \\
+ \|Q \circ \pi_F - S \circ \pi_F\| + \|S \circ \pi_F - S\| \\
< \eps.
\end{align*}
Therefore, $\diam(U \cap B_{\mathcal{P}(^2 X; \C)}) < \eps$, which proves that $\mathcal{P}(^2 X; \C)$ has the $w^*$-UKK.    
\end{proof}

The argument for bilinear forms is completely analogous. One combines the concentration result for bilinear forms with the finite-dimensional approximation argument used in the proof of \hyperref[theorem-p-convexity]{Theorem \ref*{theorem-p-convexity}}. We therefore omit the details.

\begin{theorem}\label{theorem-p-convexity bilineal} Let $X, Y$ be a complex Banach sequence lattice such that $c_{00}$ is dense in both spaces. If $X$ and $Y$ are $p$-convex with constant one for some $p>2$, then $\mathcal{L}(^2X\times Y; \C)$ (and consequently $\mathcal{L}_s(^2 X; \C)$) has the $w^*$-UKK.
\end{theorem}

The previous theorems apply to a large class of Banach sequence lattices arising from weighted $\ell_p$-type norms. We now introduce a common framework that simultaneously includes the classical sequence spaces $\ell_p$, Lorentz sequence spaces, Garling sequence spaces, and also $c_0$.

Let $p>2$ and let $\mathcal{A}$ be a family of non-negative scalar sequences. Assume that 
\begin{equation} \label{family-norm}
    \|x\|_{\mathcal{A},p}:= \sup_{a \in \mathcal{A}} \left( \sum_{k=1}^{\infty} a(k) |x(k)|^p \right)^{1/p}
\end{equation}
defines a norm on $c_{00}$ and also that its completion has the canonical unit vectors as a Schauder basis. We write $X_{\mathcal{A},p} = \overline{(c_{00}, \|\cdot\|_{\mathcal{A},p})}$.

\begin{proposition} \label{proposition-p-convex} The Banach space $X_{\mathcal{A},p}$ endowed with the norm (\ref{family-norm}) is $p$-convex with constant one. 
\end{proposition}

\begin{proof} For a finite family $x^1, \ldots, x^m$, we have 
\begin{eqnarray*}
    \left\| \left( \sum_{j=1}^m |x^j|^p \right)^{1/p} \right\|_{\mathcal{A},p}^p &=& \sup_{a \in \mathcal{A}} \sum_{k=1}^{\infty} a(k) \sum_{j=1}^m |x^j(k)|^p \\
    &\leq& \sum_{j=1}^m \sup_{a \in \mathcal{A}} \sum_{k=1}^{\infty} a(k) |x^j(k)|^p 
    = \sum_{j=1}^m \| x^j\|_{\mathcal{A},p}^p.
\end{eqnarray*}
    
\end{proof}

The proposition immediately yields several concrete examples. Indeed, each of the classical sequence spaces listed below can be realized as an example of $X_{\mathcal A,p}$ for a suitable family of weights.

\begin{corollary} The Banach spaces $\mathcal{P}(^2X;\C)$, $\mathcal{L}(^2X\times Y; \C)$ and $\mathcal{L}_s(^2 X; \C)$ have the $w^*$-UKK if $X$ and $Y$ are any (but not necessarily the same) of the following spaces.
\begin{itemize}
\item[(a)] $c_0$.
\item[(b)] $ \ell_p$ with $2 < p < \infty$.
\item[(c)] $d(w,p)$ for every Lorentz weight sequence $w$ and $2 < p < \infty$.
\item[(d)] $ g(w,p)$ for every Garling weight sequence $w$ and $2 < p < \infty$.
\end{itemize}
\end{corollary}

\begin{proof} For all the items, we apply \hyperref[proposition-p-convex]{Proposition \ref*{proposition-p-convex}} and then \hyperref[theorem-p-convexity]{Theorem \ref*{theorem-p-convexity}} or \ref{theorem-p-convexity bilineal}. For $c_0$, we take the family of weights supported on one coordinate. Let $2 < p < \infty$. For $\ell_p$, we take the constant weight one. For $d(w,p)$ we take all permutations of the weight $w$. Finally, for $g(w,p)$ we take the weights obtained by placing $w(1), w(2), \ldots$ on increasing subsequences.
\end{proof}

The class of examples can be enlarged further by considering direct sums of $p$-convex Banach sequence lattices. The next result shows that the $w^*$-UKK property is preserved under the classical $c_0$- and $\ell_s$-sum constructions.

\begin{corollary}\label{cor:direct-sums}
Let $p>2$ and let $(X_j)_{j\in\N}$ be a sequence of complex Banach sequence lattices. Suppose that, for every $j\in\N$, $c_{00}$ is dense in  $X_j$ and $X_j$ is $p$-convex with constant one. Then for $X=\left(\bigoplus_{j\in\N}X_j\right)_{c_0}$ or $\left(\bigoplus_{j\in\N}X_j\right)_{\ell_s}$ the Banach spaces
\[
\mathcal{P}\!\left(^2X;\C\right),\ 
\mathcal{L}\!\left(^2X;\C\right),\ 
\mathcal{L}_s\!\left(^2X;\C\right)
\]
have the $w^*$-UKK for every $p\le s<\infty$.
\end{corollary}

\begin{proof} For vectors $x^1, \ldots, x^N \in X$, let us write $x^k = (x_j^k)_{j=1}^{\infty}$ with $x_j^k \in X_j$ for every $j$. We have that each $X_j$ satisfies
\begin{equation*}
\left\| \left( \sum_{k=1}^N |x_j^k|^p \right)^{1/p} \right\|_{X_j} \leq \left( \sum_{k=1}^N \|x_j^k\|_{X_j}^p \right)^{1/p}.
\end{equation*}
The $c_0$-sum is $p$-convex with constant one since
\begin{eqnarray*}
\left\| \left( \sum_{k=1}^N |x^k|^p \right)^{1/p} \right\|_{c_0} &=& \sup_{j \in \N} \left\| \left( \sum_{k=1}^N |x_j^k|^p \right)^{1/p} \right\|_{X_j} \\
&\leq& \sup_{j \in \N} \left( \sum_{k=1}^N \|x_j^k\|_{X_j}^p \right)^{1/p} \\
&\leq& \left( \sum_{k=1}^N \sup_{j \in \N} \|x_j^k\|_{X_j}^p \right)^{1/p} 
= \left( \sum_{k=1}^N \|x^k\|_{c_0}^p \right)^{1/p}.
\end{eqnarray*}
The argument for the $\ell_s$-sum is similar.

\end{proof}

For simplicity in the notation, in the previous corollary we considered bilinear forms defined on the same space. However, the same argument clearly extends to bilinear forms on products of (possibly different) spaces satisfying the corresponding hypotheses.

\subsection{Applications in the real case}

The following result should be compared to \cite[Theorem 3.4]{RZSSD}. This is an application of \hyperref[real-fact-1]{Proposition~\ref*{real-fact-1}}.

\begin{proposition} \label{real-theorem-1} Let $X$ be a real or complex Banach sequence lattice such that $c_{00}$ is dense in $X$. Suppose that $X$ is $p$-convex with constant one. Let $Z$ be finite-dimensional over $\K$ and assume that $Z$ has modulus of convexity of power type $2 \leq q < p$. Then, $\mathcal{L}(X,Z) = (X \pten Z^*)^*$ has the $w^*$-UKK.
\end{proposition}

The following result should be compared to \cite[Theorem 3.7]{RZSSD}, This is an application of \hyperref[real-fact-2]{Proposition~\ref*{real-fact-2}}.

\begin{proposition}\label{real-theorem-2} Let $Z$ be a finite-dimensional uniformly convex space. Then, the Banach space $\mathcal{L}(c_0, Z) = (c_0 \pten Z^*)^*$ has the $w^*$-UKK.
\end{proposition}

It is worth noting that the finite-dimensionality in \hyperref[real-theorem-1]{Proposition~\ref*{real-theorem-1}} and \hyperref[real-theorem-2]{Proposition~\ref*{real-theorem-2}} is essential. Indeed, this is shown in the next example. This should be compared to \cite[Proposition 2]{DK}.

\begin{example} Let $X = c_0(\K)$ and $Z = \ell_2(\K)$. The space $X$ is $p$-convex with constant one for any finite $p$ and $Z$ has power type 2. However, $\mathcal{L}(c_0, \ell_2)$ fails the sequential weak-star Kadec-Klee property. Indeed, for every $n \geq 3$, let us define $T_n: c_0 \rightarrow \ell_2$ and $T: c_0 \rightarrow \ell_2$ by 
\begin{equation*}
    T_n(x) := \frac{1}{2} \left( (x_1 + x_2) e_1 + (x_1 - x_2)e_n \right) 
\end{equation*}
and 
$T(x):= \frac{1}{2}(x_1 + x_2) e_1$ for every $x = (x_k)_{k=1}^{\infty} \in c_0$. For every $x \in B_{c_0}$,  using the parallelogram identity, we obtain that $\|T_n\| \leq 1$ and since $\|T_n x\| = 1$ whenever $x_1=x_2=1$, we have that $\|T_n\| = 1$ for every $n \geq 3$. Also, $\|T\| = \sup_{x \in B_{c_0}} \frac{|x_1+x_2|}{2} = 1$. Since $\mathcal{L}(c_0, \ell_2) = (c_0 \pten \ell_2)^*$, we have that, for every $x \in c_0$ and $z \in \ell_2$, 
\begin{equation*}
    \langle T_n x, z \rangle = \frac{1}{2} (x_1 + x_2)\langle e_1, z \rangle + \frac{1}{2} (x_1 - x_2) \langle e_n, z \rangle 
\end{equation*}
for every $n \geq 3$. Since $\langle e_n, z \rangle \rightarrow 0$ as $n \rightarrow \infty$, we have that $\langle T_n x, z \rangle \rightarrow \langle Tx, z\rangle$. Therefore, $T_n \xrightarrow{w^*} T$. However, 
\begin{equation*}
    \|T_n - T\| =\sup_{x \in B_{c_0}} \frac{|x_1 - x_2|}{2} = 1.
\end{equation*}
\end{example}

\hyperref[real-fact-3]{Theorem \ref*{real-fact-3}} immediately yields the following result.

\begin{theorem} \label{real-theorem-3} Let $X$ and $Y$ be Banach sequence lattices. Suppose the following.
\begin{itemize}
\item[1.] $X$ is $p$-convex with constant one.
\item[2.] $Y$ is $s$-convex with constant one.
\item[3.] $Y^*$ has modulus of convexity of power type $2 \leq q < p$.
\item[4.] $X^*$ has modulus of convexity of power type $2 \leq r < s$.
\end{itemize}
Then, $\mathcal{L}(^2 X \times Y; \K) = (X \pten Y)^*$ has the $w^*$-UKK.
\end{theorem}

As a particular case of \hyperref[real-theorem-3]{Theorem \ref*{real-theorem-3}}, we recover \cite[Theorem 4]{DK}.

\begin{corollary} Let $2 < p,q < \infty$. Then, $\mathcal{L}(^2 \ell_p \times \ell_q; \K)$ has the $w^*$-UKK. 
\end{corollary}

\subsection{Negative results}
The assumption $p>2$ in the preceding results is essential. We conclude this section by showing that the corresponding weak-star Kadec--Klee properties fail when $1<p\le2$.

\begin{example}\label{prop:lp-small-p} 
Let $1<p\leq2$. Then $\mathcal{P}(^2\ell_p;\C)$ fails the sequential $w^*$-KK. For $m\geq2$, define $P_m(x)=x_1^2+x_m^2$ and $P(x)=x_1^2$. Since $2\geq p$, for every $x\in B_{\ell_p}$, we have that
\begin{equation*}
|P_m(x)|\leq|x_1|^2+|x_m|^2\leq|x_1|^p+|x_m|^p\leq1.
\end{equation*}
Hence $\|P_m\|=\|P\|=1$. For every fixed $x\in\ell_p$, $x_m\longrightarrow0$ as $m\rightarrow\infty$ and therefore $P_m\xrightarrow{w^*}P$ by \hyperref[lem:wstar]{Lemma~\ref*{lem:wstar}}. Finally, for every $m\geq2$, $\|P_m-P\|=\|x_m^2\|=1$.
\end{example}

\begin{example} \label{ex:Lorentz-fails-1-p-2} Let $1<p\leq2$. Then, $\mathcal{P}(^2 d(w,p); \C)$ does not have the sequential $w^*$-KK. For $m \geq 2$, we define 
\begin{equation*}
    Q_m(x):= x(1)^2 + [w(2)]^{2/p} x(m)^2
\end{equation*}
for every $x \in d(w,p)$. Define also $Q(x) = x(1)^2$ for every $x \in d(w,p)$. We start by showing that $\|Q_m\| = \|Q\| = 1$ for every $m \geq 2$. Indeed, fix $x \in d(w,p)$. Since $(x^{\star}(k))_{k=1}^{\infty}$ is obtained by arranging all the numbers $(|x(k)|)_{k=1}^{\infty}$ in decreasing order, one can show that
\begin{equation} \label{lorentz-eq1}
   \|x\|_{d(w,p)}^p \geq \max \{ |x(1)|, |x(m)| \}^p + w(2) \min \{ |x(1)|, |x(m)| \}^p. 
\end{equation}
Now, since $2/p \geq 1$, for every $\alpha, \beta \geq 0$, 
\begin{equation} \label{lorentz-eq2}  
\alpha^{2/p} + \beta^{2/p} \leq (\alpha + \beta)^{2/p}.
\end{equation} 
We consider two cases. Suppose first that $|x(1)| \geq |x(m)|$. In this case, we have
\begin{eqnarray*}
|Q_m(x)| &\leq& |x(1)|^2 + [w(2)]^{2/p} |x(m)|^2 \\
&=& (|x(1)|^p)^{2/p} + (w(2) |x(m)|^p)^{2/p} \\
&\stackrel{(\ref{lorentz-eq2})}{\leq}& (|x(1)|^p + w(2) |x(m)|^p)^{2/p} 
\stackrel{(\ref{lorentz-eq1})}{\leq} \|x\|_{d(w,p)}^2.
\end{eqnarray*}
Suppose now that $|x(m)| \geq |x(1)|$. Since $0 < w(2)^{2/p} \leq 1$, we have that 
\begin{equation*}
   |x(1)|^2 + [w(2)]^{2/p} |x(m)|^2 \leq |x(m)|^2 + [w(2)]^{2/p} |x(1)|^2
\end{equation*}
and, therefore, using once again (\ref{lorentz-eq2}) and then (\ref{lorentz-eq1}), we get that $|Q_m(x)| \leq \|x\|_{d(w,p)}^2$. This implies that $\|Q_m\| \leq 1$ for every $m \geq 2$. Since $Q_m(e_1) = 1 = \|e_1\|_{d(w,p)}$, we have that $\|Q_m\| = 1$ for every $m \geq 2$. Similarly, we can prove that $\|Q\| = 1$.

On the other hand, for every $x \in d(w,p) \subseteq c_0$, we have that $Q_m(x) \rightarrow Q(x)$ and then $Q_m \xrightarrow{w^*} Q$. However, for every $m \geq 2$, we have that $(Q_m - Q)(x) = [w(2)]^{2/p} x(m)^2$ for every $x \in d(w,p)$. Since $x \mapsto x(m)^2$ is a 2-homogeneous polynomial of norm-one and $|x(m)| \leq \|x\|_{d(w,p)}$ gives an upper bound while evaluating at $e_m$ gives equality, we have that $\|Q_m - Q\| = w(2)^{2/p}$ for every $m \geq 2$. 
\end{example}

For the Garling space, one does not have the two cases from \hyperref[ex:Lorentz-fails-1-p-2]{Example \ref*{ex:Lorentz-fails-1-p-2}} and the argument is simpler. For that reason, we omit it.

\begin{example} \label{ex:Garling-fails-1-p-2} Let $1 < p \leq 2$. Then $\mathcal{P}(^2 g(w,p); \C)$ fails the sequential $w^*$-KK.
\end{example}

\subsection{Consequences and characterizations}
In \cite{DK} the authors show that for real spaces, $\mathcal{L}(^2 \ell_p\times\ell_q;\R)$ has the $w^*$-UKK if and only if $2<p,q < \infty$. Combining the positive results established above with the counterexamples from the previous subsection (see \hyperref[prop:lp-small-p]{Examples \ref*{prop:lp-small-p}}, \ref{ex:Lorentz-fails-1-p-2} and \ref{ex:Garling-fails-1-p-2}), we also obtain a complete characterization of the $w^*$-UKK for $2$-homogeneous polynomials on the classical sequence spaces considered in this paper.

\begin{corollary}\label{cor:lp-lorentz-garling-classification}
Let $1<p<\infty$, and let $w=(w(k))_{k=1}^{\infty}$ be a decreasing sequence of positive numbers satisfying $w(1)=1$, $\lim_{k\to\infty}w(k)=0$ and $\sum_{k=1}^{\infty}w(k)=\infty$. For each of the complex spaces $X=\ell_p(\C)$, $X=d(w,p)$ or $X=g(w,p)$, the following assertions are equivalent.
\begin{enumerate}[\rm (a)]
\item $2 < p < \infty$.
\item $\mathcal{P}(^2X;\C)$ has the $w^*$-UKK.
\item $\mathcal{P}(^2X;\C)$ has the sequential $w^*$-UKK.
\item $\mathcal{P}(^2X;\C)$ has the $w^*$-KK.
\item $\mathcal{P}(^2X;\C)$ has the sequential $w^*$-KK.
\end{enumerate}
\end{corollary}

\section{Negative result for higher degrees}\label{sec: higher degrees}

In this section we establish several general results showing that, in all higher degrees, weak-star Kadec--Klee properties fail in a very robust way. The main tool is the Josefson-Nissenzweig theorem (see, for instance, \cite{Josefson, Nissenzweig}), which provides weak-star null sequences of norm-one functionals. By combining such sequences with suitable polynomial and multilinear perturbations, we construct bounded weak-star convergent sequences that remain uniformly separated in norm.

We begin by stating two simple lemmas that will be used repeatedly throughout this section.

\begin{lemma}\label{JN-lemma} Let $X$ be an infinite-dimensional Banach space and let $x_0 \in X \setminus \{0\}$. There exists a sequence $(g_m) \subseteq S_{X^*}$ such that $g_m \xrightarrow{w^*} 0$ and $g_m(x_0) = 0$ for every $m \in \N$.
\end{lemma}

\begin{proof} Let $Q:X\longrightarrow X/\K x_0$ be the quotient map. Since $X/\K x_0$ is infinite-dimensional, the Josefson--Nissenzweig theorem applied to the quotient gives a sequence $(\widetilde g_m)\subseteq S_{(X/\K x_0)^*}$ such that $\widetilde g_m\xrightarrow{w^*}0$. Define $g_m=Q^*(\widetilde g_m)=\widetilde g_m\circ Q$. Then $(g_m)\subseteq S_{X^*}$, $g_m\xrightarrow{w^*}0$ and $g_m(x_0)=0$ for every $m\in\N$.
\end{proof}

\begin{lemma} \label{lemma: pol complex} Let $X$ be a Banach space with $\dim X \geq 2$. Then, there exist $f,g \in X^*$ such that $\|f\| = 1$, $g\not=0$ and 
\begin{equation} \label{lemma: pol complex1}
    |f(x)|^2 + |g(x)|^2 \leq \|x\|^2 
\end{equation}
for every $x \in X$. 
\end{lemma}

\begin{proof} Take an injective linear operator $T: \ell_2^2 \rightarrow X^*$ and normalize it so that $\|T\| = 1$. Take $a \in S_{\ell_2^2}$ to be such that $\|T(a)\| = 1$. Composing with a rotation on $\ell_2^2$, we may assume $a=e_1$.  Define then $f:= Te_1$ and $g:= Te_2$. Then, $\|f\| = \|Te_1\| = 1$ and also $g\not=0$ as $T$ is injective. Finally, for every $x \in X$, we have that
\begin{eqnarray*}
    \left( |f(x)|^2 + |g(x)|^2 \right)^{1/2} &=& \sup_{|\alpha|^2+|\beta|^2\leq1} |\alpha f(x) + \beta g(x)| \\
    &\leq& \sup_{|\alpha|^2+|\beta|^2\leq1} \|\alpha f + \beta g\| \|x\| \\
    &=& \sup_{|\alpha|^2+|\beta|^2\leq1} \|T(\alpha e_1 + \beta e_2)\| \|x\| \\
    &\leq& \sup_{|\alpha|^2+|\beta|^2\leq1}\|T\| \| \alpha e_1 + \beta e_2\|_2 \|x\| \leq \|x\|.
\end{eqnarray*}
\end{proof}

\subsection{Homogeneous polynomials}

We first study spaces of homogeneous polynomials. The results below show that, apart from the $2$-homogeneous complex case, weak-star Kadec--Klee properties fail in every degree. The real and complex settings require slightly different constructions. In the real case the argument already works for degree two, whereas in the complex case the obstruction appears from degree three onwards.

\begin{theorem} \label{prop:pol real}
Let $X$ be an infinite-dimensional real Banach space and let $n\geq2$. Then $\mathcal{P}(^nX;\R)$ fails the $w^*$-KK.
\end{theorem}

\begin{proof}
Fix $n\geq2$. Let $f \in S_{X^*}$ be such that $f(x_0) = 1$ for some $x_0\in S_X$. Let $(g_m) \subseteq S_{X^*}$ be the sequence from \hyperref[JN-lemma]{Lemma \ref*{JN-lemma}}. For every $x\in X$, define $P(x)=f(x)^n$ and, for every $m\in\N$,
\begin{equation*}
P_m(x)=f(x)^{n-2}\bigl(f(x)^2-g_m(x)^2\bigr).
\end{equation*}
For every $x\in B_X$, both $f(x)$ and $g_m(x)$ are real numbers in $[-1,1]$. Therefore,
\begin{equation*}
|P_m(x)|\leq |f(x)|^{n-2}|f(x)^2-g_m(x)^2|\leq1.
\end{equation*}
Since $P_m(x_0)=P(x_0)=1$, we have $\|P\|=\|P_m\|=1$ for every $m\in\N$. Notice also that $P_m(x)\longrightarrow P(x)$ for every $x\in X$. By \hyperref[lem:wstar]{Lemma~\ref*{lem:wstar}}, $P_m\xrightarrow{w^*}P$ in $\mathcal{P}(^nX;\R)$. We prove that norm convergence fails.  By the results in \cite{BST} we have that
\begin{equation*}
\|P_m-P\|
=\|f^{n-2}g_m^2\| \geq \frac{\|f\|^{n-2} \|g_m\|^2}{n^n} =\frac{ 1}{n^n}.
\end{equation*}
\end{proof}

This result also shows that the polynomial and bilinear settings behave quite differently over the real scalars. Indeed, Dilworth and Kutzarova \cite{DK} proved that $\mathcal{L}(^2\ell_p;\R)$ has the $w^*$-UKK whenever $2<p<\infty$. Recall that in this case, $\ell_p\pten \ell_p$ is separable and reflexive (see \cite[Corollary 4.24]{ryan2002introduction}). Consequently, the same is true for the symmetric subspace $\mathcal{L}_s(^2\ell_p;\R)$. Since, by \hyperref[prop:pol real]{Theorem \ref*{prop:pol real}}, $\mathcal{P}(^2\ell_p;\R)$ never has the $w^*$-KK, it follows that neither the $w^*$-KK for bilinear forms nor its symmetric counterpart implies the corresponding property for homogeneous polynomials, at least in the real setting. Whether the converse implication holds in the complex setting, namely whether the $w^*$-KK of $\mathcal{P}(^2X;\C)$ forces the same property for $\mathcal{L}(^2X;\C)$ or $\mathcal{L}_s(^2X;\C)$, remains open.

The previous argument can be adapted to prove a complex version of \hyperref[prop:pol real]{Theorem \ref*{prop:pol real}} but now for $n \geq 3$. In fact, the next proof works for both real and complex Banach spaces.

\begin{theorem} \label{prop:pol complex} Let $X$ be a complex infinite-dimensional Banach space and $n\geq 3$. Then, $\mathcal{P}(^n X; \C)$ fails the sequential $w^*$-KK.
\end{theorem}

\begin{proof} Consider $f,g \in X^*$ as in \hyperref[lemma: pol complex]{Lemma \ref*{lemma: pol complex}}. By the Josefson–Nissenzweig theorem there is $(g_m) \subseteq S_{X^*}$ $w^*$-convergent to $0$. Let $n \geq 3$ be fixed. Define 
\begin{equation*}
    P(x):= f(x)^n \ \ \ \mbox{and} \ \ \ P_m(x):= f(x)^n + g(x)^2 g_m(x)^{n-2}, \forall m \in \N 
\end{equation*}
for every $x \in X$. For every $x \in B_X$, we have that 
\begin{eqnarray*}
|P_m(x)| &\leq& |f(x)|^n + |g(x)|^2|g_m(x)|^{n-2} \\
&\stackrel{(n\geq3)}{\leq}& |f(x)|^2 + |g(x)|^2 \\
&\stackrel{(\eqref{lemma: pol complex1})}{\leq}& \|x\|^2 \leq 1.
\end{eqnarray*}
So, $\|P_m\| \leq 1$ for every $m \in \N$. Now, since $\|f\| = 1$, there exists a sequence
$(x_j)\subseteq S_X$ such that $|f(x_j)| \rightarrow 1$ as $j \rightarrow \infty$. Then, by \eqref{lemma: pol complex1}, we must have that $g(x_j) \rightarrow 0$ as $j \rightarrow \infty$. So, for a fixed $m$, we obtain
\begin{equation*}
    |P_m(x_j)| \geq |f(x_j)|^n - |g(x_j)|^2 \rightarrow 1 
\end{equation*}
as $j \rightarrow \infty$. So, $\|P_m\| = 1$ for every $m \in \N$. Clearly, we have that $\|P\|=1$.

Now, for every $x \in X$, $g_m(x) \rightarrow 0$ as $m \rightarrow \infty$ and since $n-2 \geq 1$, we get that $P_m(x) \rightarrow P(x)$ for every $x \in X$ as $m \rightarrow \infty$. Since $(P_m)$ is bounded, by \hyperref[lem:wstar]{Lemma \ref*{lem:wstar}}, we get that $P_m \xrightarrow{w^*} P$.

It remains to separate $P_m$ uniformly from $P$. By the results in \cite{BST} we have

    $$\|P_m - P\| = \|g^2 g_m^{n-2}\| \geq \frac{\|g\|^2 \|g_m\|^{n-2}}{n^n} = \frac{\|g\|^2}{n^n}.$$
\end{proof}

\hyperref[prop:pol real]{Theorems \ref*{prop:pol real}} and \ref{prop:pol complex} explain why, in the previous sections, we restricted our attention to the Banach space $\mathcal{P}(^2X;\mathbb{C})$, since the $w^*$-KK property already fails for every degree $n\ge2$ over the real field and for every degree $n\ge3$ over the complex field.

Before addressing the multilinear case, we first record a simple stability property showing that any weak-star Kadec--Klee type property of this polynomial space is necessarily inherited by the dual space $X^*$.

\begin{proposition} \label{prop:stability}
Let $X$ be a complex Banach space. Then the following assertions hold.
\begin{enumerate}[\rm (a)]
\item If $\mathcal{P}(^2X;\mathbb{C})$ has the (respectively, sequential) $w^*$-UKK, then $X^*$ has the (respectively, sequential) $w^*$-UKK.
\item If $\mathcal{P}(^2X;\mathbb{C})$ has the (respectively, sequential) $w^*$-KK, then $X^*$ has the (respectively, sequential) $w^*$-KK.
\end{enumerate}
\end{proposition}

\begin{proof} We only prove the net part of item (a), since the remaining
cases are entirely analogous. Assume that $\mathcal{P}(^2 X; \C)$ has the $w^*$-UKK. Let $\eps \in (0,2]$ be given and let $(f_{\alpha}) \subseteq B_{X^*}$ be such that $f_{\alpha} \xrightarrow{w^*} f$ and $\sep(f_{\alpha}) \geq \eps$. We must obtain a bound $\|f\| \leq 1 - \delta$, where $\delta > 0$ depends only on $\eps$. Suppose that we have 
\begin{equation*}
    \|f\| > 1 - \frac{\eps}{16}
\end{equation*}
as otherwise we would have the required conclusion already. Let us choose $y \in S_X$ such that, after multiplying it by a modulus one complex scalar if necessary, $f(y) > \|f\| - \frac{\eps}{32}$. This means that $f(y) > 1 - \frac{3 \eps}{32}$. Now, as $f_{\alpha}(y) \rightarrow f(y)$, there exists $\alpha_0$ such that, whenever $\alpha \geq \alpha_0$, we have that 
\begin{equation*}
    |f_{\alpha}(y) - f(y)| < \frac{\eps}{32}
\end{equation*}
which implies that $|f_{\alpha}(y)| > 1 - \frac{\eps}{8}$ for every $\alpha \geq \alpha_0$. At the same time, also for $\alpha, \beta \geq \alpha_0$, we have that $|f_{\alpha}(y) - f_{\beta}(y)| < \frac{\eps}{16}$.

Now, for every $\alpha$, we define the 2-homogeneous polynomial $P_{\alpha}: X \rightarrow \C$ by $P_{\alpha}(x):= f_{\alpha}(x)^2$ for every $x \in X$. Take $\alpha < \beta$ to be such that $\alpha, \beta \geq \alpha_0$. Their associated bilinear forms satisfy
\begin{eqnarray*}
(\check{P}_{\alpha} - \check{P}_{\beta})(x,y) &=& f_{\alpha}(x) f_{\alpha}(y) - f_{\beta}(x)f_{\beta}(y) \\
&=& (f_{\alpha}(x) - f_{\beta}(x)) f_{\alpha}(y) + f_{\beta}(x) (f_{\alpha}(y) - f_{\beta}(y))
\end{eqnarray*}
for every $x,y \in X$. Taking the supremum over $x \in B_X$ yields 
\begin{equation*}
    \| \check{P}_{\alpha} - \check{P}_{\beta}\| \geq \|f_{\alpha} - f_{\beta}\| |f_{\alpha}(y)| - |f_{\alpha}(y) - f_{\beta}(y)| \geq \eps \left(1 - \frac{\eps}{8} \right) - \frac{\eps}{16} \geq \frac{\eps}{2}.
\end{equation*}
as $\sep(f_{\alpha}) \geq \eps$ and $0 < \eps \leq 2 $. The polarization formula (see \cite[Section 2.1]{floret1997natural}) then gives 
\begin{equation*}
    \|P_{\alpha} - P_{\beta}\| \geq \frac{1}{2} \|\check{P}_{\alpha} - \check{P}_{\beta}\| \geq \frac{\eps}{4}.
\end{equation*}
Moreover, we have that $P_{\alpha} \xrightarrow{w^*}P$ with $P(x):= f(x)^2$ for every $x \in X$ and $\|P\| = \|f\|^2$. Since $\mathcal{P}(^2 X; \C)$ has the $w^*$-UKK, there exists a positive number $\delta'>0$ depending only on $\eps$ such that $\|f\|^2 = \|P\| \leq 1 - \delta'$. This means that $\|f\| \leq \sqrt{1-\delta'}$. Thus, we have the desired conclusion with 
$$\delta=\min \left\{\frac{\eps}{16} , 1-\sqrt{1-\delta'}\right\}.$$
\end{proof}

The following remark is relevant to our discussion, as it shows that determining whether $\mathcal{P}(^2X;\mathbb{C})$ has a weak-star Kadec-Klee property is a nontrivial problem. Indeed, in general, this cannot be reduced to the corresponding property in $X^*$.

\begin{remark} Neither converse of \hyperref[prop:stability]{Proposition \ref*{prop:stability}} above is true in general. Indeed, $\mathcal{P}(^2 \ell_2; \C)$ fails the sequential $w^*$-KK. See \hyperref[prop:lp-small-p]{Example \ref*{prop:lp-small-p}} or \hyperref[ex:Lorentz-fails-1-p-2]{Example \ref*{ex:Lorentz-fails-1-p-2}}.    
\end{remark}

\subsection{(Symmetric) multilinear forms}

We now turn to multilinear forms. As in the polynomial setting, our first goal is to determine in which degrees weak-star Kadec--Klee properties may possibly hold. We shall prove that, for both multilinear and symmetric multilinear forms, these properties fail in every degree $n\ge3$. Consequently, only the bilinear case remains of interest.

Unlike the polynomial setting, however, the real and complex cases exhibit different behavior. Indeed, Dilworth and Kutzarova \cite{DK} proved that $\mathcal{L}(^2\ell_p(\mathbb R)\times\ell_q(\mathbb R);\mathbb R)$ has the $w^*$-UKK whenever $2<p,q<\infty$, showing that bilinear forms behave differently from $2$-homogeneous polynomials.

Our first result is the following.

\begin{theorem} \label{theorem:multinear-symmetric-fails-n-bigger-3} Let $X$ be an infinite-dimensional Banach space over $\K$ and let $n\geq 3$. Then, $\mathcal{L}_s(^n X; \K)$ fails the sequential $w^*$-KK.
\end{theorem}

\begin{proof} Let $n \geq 3$. Choose $f, g \in X^*$ as in \hyperref[lemma: pol complex]{Lemma \ref*{lemma: pol complex}}. By the Josefson–Nissenzweig theorem there exists $(h_m) \subseteq S_{X^*}$ such that $h_m \xrightarrow{w^*} 0$. For every $m \in \N$, define 
\begin{equation*}
A_m(x_1,\ldots,x_n)= \prod_{k=1}^{n} f(x_k) + \frac{1}{\binom{n}{2}} \sum_{1\leq i<j\leq n} g(x_i)g(x_j) \prod_{\substack{1\leq k\leq n\\ k\neq i,j}} h_m(x_k)
\end{equation*}
for every $x_1, \ldots, x_n \in X$. Notice that $\binom{n}{2}$ is exactly the number of pairs $(i,j)$ satisfying $1 \leq i < j \leq n$. Also, for $x_1, \ldots, x_n \in X$, define 
\begin{equation*}
    A(x_1, \ldots, x_n):= \prod_{k=1}^n f(x_k).
\end{equation*}
These are symmetric $n$-linear forms on $X^n$. Notice now that, for every $x_1, \ldots, x_n \in B_X$ and for every pair $i < j$, by using (\ref{lemma: pol complex1}) and the Cauchy-Schwarz inequality, we have that
\begin{eqnarray*}
    \left| \prod_{k=1}^n f(x_k) + g(x_i)g(x_j) \prod_{\substack{1\leq k\leq n\\ k\neq i,j}} h_m(x_k) \right| &\leq& |f(x_i)f(x_j)| + |g(x_i)g(x_j)| \\
    &\leq& (|f(x_i)|^2 + |g(x_i)|^2)^{1/2}(|f(x_j)|^2+|g(x_j)|^2)^{1/2} \\
    &\leq& 1.
\end{eqnarray*}
Since, for every $m \in \N$, $A_m(x_1, \ldots, x_n)$ is the average of these expressions over all pairs $i<j$, it follows that $\|A_m\| \leq 1$ for every $m \in \N$. Now, let us choose $(u_j) \subseteq B_X$ to be such that $|f(u_j)| \rightarrow 1$ as $j \rightarrow \infty$. Once again by (\ref{lemma: pol complex1}), we must have that $g(u_j) \rightarrow 0$ as $j \rightarrow \infty$. For each fixed $m \in \N$,
\begin{eqnarray*}
|A_m(u_j, \ldots, u_j)| &=& |f(u_j)^n + g(u_j)^2 h_m(u_j)^{n-2}| \\
&\geq& |f(u_j)|^n - |g(u_j)|^2 \rightarrow 1
\end{eqnarray*}
as $j \rightarrow \infty$. This means that $\|A_m\| = 1$ for every $m \in \N$. Clearly, we have that $\|A\| = 1$. Now, for every fixed $x_1, \ldots, x_n \in X$, we have that 
\begin{equation*}
    (A_m - A)(x_1, \ldots, x_n) = \frac{1}{\binom{n}{2}} \sum_{1\leq i<j\leq n} g(x_i)g(x_j) \prod_{\substack{1\leq k\leq n\\ k\neq i,j}} h_m(x_k).
\end{equation*}
As $n \geq 3$, every summand contains at least one factor $h_m(x_k)$, this means that $A_m(x_1, \ldots, x_n) \rightarrow A(x_1, \ldots, x_n)$ as $m \rightarrow \infty$. As the sequence $(A_m)$ is bounded, we have that $A_m \xrightarrow{w^*} A$ in $\mathcal{L}_s(^n X; \K)$.

Finally, since 
\begin{equation*}
    \|A_m - A\| \geq \sup_{\|x\|\leq 1} |g(x)|^2|h_m(x)|^{n-2},
\end{equation*}
we have that $ \|A_m - A\|$ is greater than or equal to the norm of the polynomial $g^2h_m^{n-2}$ which, by the results in \cite{BST} is at least 
$$\frac{\|g\|^2\|h_m\|^{n-2}}{n^n}=\frac{\|g\|^2}{n^n}>0.$$ In particular $A_m$ does not converge to $A$ in norm.
\end{proof}

As a consequence of \hyperref[prop:multi implies simetric]{Fact \ref*{prop:multi implies simetric}} and \hyperref[theorem:multinear-symmetric-fails-n-bigger-3]{Theorem \ref*{theorem:multinear-symmetric-fails-n-bigger-3}}, we immediately obtain a negative result for $n$-linear forms on an infinite-dimensional Banach space $X$. However, with a little additional work, we can prove a stronger and more general result.

\begin{theorem}\label{theorem:multilinear-fails-n-bigger-3}
Let $X_1,\ldots,X_n$ be Banach spaces over $\K$ of dimension at least $2$, where $n\ge3$. If at least one of the spaces $X_1,\ldots,X_n$ is infinite-dimensional, then
\[
\mathcal{L}(^nX_1\times\cdots\times X_n;\K)
\]
fails the $w^*$-KK.
\end{theorem}
\begin{proof}
Without loss of generality, assume that $X_n$ is infinite-dimensional.

For $i=1,2$, choose $f_{i},g_i\in X_{i}^*$ as in \hyperref[lemma: pol complex]{Lemma~\ref*{lemma: pol complex}}. For $i=3,\ldots n$ take any  $f_{i}\in S_{X^*_i}$. By the Josefson–Nissenzweig theorem there is $(h_m)\subseteq S_{X_n^*}$ , so that
\(
h_m\xrightarrow{w^*}0
\).

Define
\[
A(x_1,\ldots,x_n)
=\prod_{i=1}^{n}f_i(x_i),
\]
and, for every $m\in\N$,
\[
A_m(x_1,\ldots,x_n)
=
A(x_1,\ldots,x_n)
+
\Bigl(\prod_{i=3}^{n-1}f_i(x_i)\Bigr)
g_{1}(x_{1})g_{2}(x_{2})h_m(x_n).
\]

Clearly $A$ and $A_m$ are $n$-linear forms. Moreover, for every
$x_1,\ldots,x_n\in B_{X_j}$,
\begin{align*}
|A_m(x_1,\ldots,x_n)|
&\leq
\bigl|f_{1}(x_{1})f_{2}(x_{2})\bigr| +
\bigl|g_{1}(x_{1})g_{2}(x_{2})\bigr| \\
&\leq
\Bigl(
    |f_{1}(x_{1})|^2
    +|g_{1}(x_{1})|^2
\Bigr)^{1/2} 
\Bigl(
    |f_{2}(x_{2})|^2
    +|g_{2}(x_{2})|^2
\Bigr)^{1/2} \\
&\leq 1.
\end{align*}
where we used \hyperref[lemma: pol complex]{Lemma~\ref*{lemma: pol complex}} and the Cauchy--Schwarz inequality. Hence $\|A_m\|\le1$ for every $m$. Now, for each $i=1,\ldots,n$, let us choose $(u_j^i) \subseteq B_{X_i}$ to be such that $|f_i(u_j^i)| \rightarrow 1$ as $j \rightarrow \infty$. By (\ref{lemma: pol complex1}), we must have that $g_i(u_j^{i}) \rightarrow 0$ as $j \rightarrow \infty$ for $i=1,2$. For each fixed $m \in \N$,
\begin{eqnarray*}
|A_m(u_j^1, \ldots, u_j^n)| &\geq& \prod_{i=1}^n|f_i(u_j^i)| - |g_{1}(u_j^{1})g_{2}(u_j^{2})| \rightarrow 1
\end{eqnarray*}
as $j \rightarrow \infty$. This means that $\|A_m\| = 1$ for every $m \in \N$. 

It is easy to see that $\|A\|=1$. Since $h_m\xrightarrow{w^*}0$, we have
$A_m(x_1,\ldots,x_n)\to A(x_1,\ldots,x_n)$ for every fixed
$(x_1,\ldots,x_n)$. As $(A_m)$ is bounded, \hyperref[lem:wstar]{Lemma~\ref*{lem:wstar}}
implies that $A_m\xrightarrow{w^*}A$.

Finally, since
\[
\|A_m-A\| = \left(\prod_{i=3}^{n-1}\|f_i\|\right)
\|g_{1}\| \|g_{2}\| \|h_m\|= 
\|g_{1}\| \|g_{2}\|,
\]
$(A_m)$ does not converge to $A$ in norm, and
$\mathcal L(^nX_1\times\cdots\times X_n;\K)$ fails the weak-star
Kadec--Klee property.
\end{proof}

\subsection{The vector-valued case}

One may ask whether weak-star Kadec-Klee properties can occur in the vector-valued spaces $\mathcal{P}(^nX;Z^*)$ or $\mathcal{L}_s(^nX;Z^*)$ when at least one of the Banach spaces is infinite-dimensional. Note that the case $\mathcal{L}(^nX;Z^*)=\mathcal{L}(^{n+1}X\times\cdots X\times Z;\K) $ was already covered in \hyperref[theorem:multilinear-fails-n-bigger-3]{Theorem \ref*{theorem:multilinear-fails-n-bigger-3}}. In this short subsection, we show that, apart from the natural one-dimensional exceptions, the Josefson--Nissenzweig theorem always yields negative results.

Since we are dealing with weak-star properties, it is convenient to recall the corresponding preduals in the vector-valued setting. Let $n \geq 2$. Let $X$ be a Banach space over $\K$ and let $Y = Z^*$ be a dual Banach space with a fixed predual $Z$. The standard projective and symmetric-projective tensor linearization give isometric identifications as follows 
\begin{equation*}
\mathcal{P}(^n X;Z^*) = \left( \left(\widetilde{\otimes}_{\pi_s}^{n,s} X\right) \pten Z \right)^* \ \ \ \mbox{and} \ \ \ \mathcal{L}_s(^n X; Z^*) = \left( \left( \widetilde{\otimes}_{\pi}^{n,s} X \right)\pten Z \right)^*.
\end{equation*}

As in \hyperref[lem:wstar]{Lemma \ref*{lem:wstar}}, on norm-bounded nets, weak-star convergence in $\mathcal{P}(^n X; Z^*)$ is equivalent to $\langle P_{\alpha}(x), z \rangle \rightarrow \langle P(x), z \rangle$ for every $x \in X$ and $z \in Z$. Similarly, bounded weak-star convergence in $\mathcal{L}_s(^nX; Z^*)$ is equivalent to $\langle A_{\alpha}(x_1, \ldots, x_n), z \rangle \rightarrow \langle A(x_1, \ldots, x_n), z \rangle$ for every $x_1, \ldots, x_n \in X$ and $z \in Z$.

We next show that, unless one of the spaces involved is one-dimensional, the same obstruction provided by the Josefson--Nissenzweig theorem already rules out weak-star Kadec--Klee properties in the vector-valued setting.

\begin{proposition} Let $X, Z$ be Banach spaces. Let $\dim X \geq 2$ and let $Z^*$ be infinite-dimensional. Then, for every $n \geq 2$, the space $\mathcal{P}(^n X; Z^*)$ fails the sequential $w^*$-KK.
\end{proposition}

\begin{proof} Fix $n \geq 2$. Choose $f, g \in X^*$ as in \hyperref[lemma: pol complex]{Lemma \ref*{lemma: pol complex}}. By the Josefson--Nissenzweig theorem, there is a sequence $(y_m) \subseteq S_{Z^*}$ such that $y_m \xrightarrow{w^*} 0$. Fix any $y_0 \in S_{Z^*}$ and define, for every $m \in \N$,
\begin{equation*}
    P_m(x) := f(x)^n y_0 + f(x)^{n-2} g(x)^2 y_m 
\end{equation*}
and $P(x) := f(x)^n y_0$ for every $x \in X$. As before, one can show that $\|P_m\| = \|P\| = 1$ for every $m \in \N$. Also, $P_m \xrightarrow{w^*} P$ but $\|P_m - P\| = \|f^{n-2} g^2\| > 0$ for every $m \in \N$.
\end{proof}

The previous proposition excludes the case in which the range dual space is infinite-dimensional. By interchanging the roles of the domain and range, one obtains the complementary statement.

\begin{proposition} Let $X, Z$ be Banach spaces. Suppose $X$ is infinite-dimensional and $\dim Z \geq 2$. Then, for every $n \geq 2$, the space $\mathcal{P}(^n X; Z^*)$ fails the sequential $w^*$-KK.
\end{proposition}

Analogously, one can prove these results for $\mathcal{L}_s(^nX; Z^*)$.
\begin{proposition} Let $X, Z$ be Banach spaces. Let $n \geq 2$. If $X$ is infinite-dimensional and $\dim Z \geq 2$, or $\dim X \geq 2$ and $Z$ is infinite-dimensional, then $\mathcal{L}_s(^n X; Z^*)$ fails the sequential $w^*$-KK.
\end{proposition}

In view of the preceding results, the only non-trivial spaces that remain to be studied are $\mathcal{P}(^2X;\mathbb{C})$, $\mathcal{L}(X\times Y;\mathbb{K})$, and $\mathcal{L}_s(^2X;\mathbb{K})$, which we already did in the previous sections. These are therefore the only cases deserving investigation.

\section{Further remarks}\label{sec:further-remarks}

In this final section we collect several remarks that complement the main results of the paper.

We first discuss the optimality of the concentration theorems established in \hyperref[sec: concentration]{Section~\ref*{sec: concentration}}, showing that each of their assumptions is necessary in view of the negative results obtained throughout the paper. We then point out that our arguments extend to the non-separable spaces $c_0(\Gamma)$ and $\ell_p(\Gamma)$ for an arbitrary index set $\Gamma$. Finally, we describe some consequences of our results concerning strong subdifferentiability of the norm of a Banach space.

\subsection*{On the optimality of the concentration results}

The concentration results proved in \hyperref[sec: concentration]{Section~\ref*{sec: concentration}} constitute the main technical ingredient in our positive weak-star Kadec--Klee results. In particular, they imply the $w^*$-UKK property for the corresponding spaces of homogeneous polynomials and bilinear forms.

From this point of view, the negative examples obtained throughout the paper also provide information about the limitations of these concentration principles. Indeed, if a concentration theorem of the same nature held under substantially weaker hypotheses, then the arguments of \hyperref[sec: quadratic and bilinear]{Section~\ref*{sec: quadratic and bilinear}} would still yield the corresponding $w^*$-UKK, contradicting the counterexamples established earlier.

More precisely, the examples of spaces not having the $w^*$-UKK show the following limitations.

\begin{itemize}
\item[(a)] The restriction to complex scalars is essential in the polynomial setting. Indeed, \hyperref[prop:pol real]{Theorem~\ref*{prop:pol real}} shows that, over the real field, $\mathcal P(^nX;\mathbb R)$ never has the $w^*$-KK for any infinite-dimensional Banach space whenever $n\ge2$. Consequently, no concentration result analogous to \hyperref[lemma-p-convexity]{Theorem~\ref*{lemma-p-convexity}} can hold in the real setting.

\item[(b)] The restriction to degree two is also essential. \hyperref[prop:pol complex]{Theorem~\ref*{prop:pol complex}} shows that, for complex Banach spaces, $\mathcal P(^nX;\mathbb C)$ fails the sequential $w^*$-KK whenever $n\ge3$. Therefore, the concentration phenomenon cannot extend to higher-degree homogeneous polynomials.

\item[(c)] The assumption of $p$-convexity with $p>2$ is likewise indispensable. Indeed, \hyperref[prop:lp-small-p]{Example~\ref*{prop:lp-small-p}} and \hyperref[ex:Lorentz-fails-1-p-2]{Example~\ref*{ex:Lorentz-fails-1-p-2}} show that $\mathcal P(^2\ell_p;\mathbb C)$ and $\mathcal P(^2d(w,p);\mathbb C)$ fail even the sequential $w^*$-KK whenever $1<p\le2$. Thus, no concentration theorem of the same type can hold without some geometric hypothesis comparable to $p$-convexity.

\item[(d)] In \hyperref[thm: conc-bilin]{Theorem \ref*{thm: conc-bilin}}, the assumption that the scalars are complex is also unavoidable. Indeed, Dilworth and Kutzarova showed that the real space $\mathcal L(^2c_0(\mathbb R);\mathbb R)$ fails this property. Consequently, even assuming the relevant convexity hypothesis, one cannot expect a concentration theorem of the type proved here over the real scalars.
\end{itemize}

Altogether, these examples indicate that, although the hypotheses of our concentration theorems may perhaps be reformulated or generalized in certain directions, none of them can be removed entirely. Each assumption rules out a genuine obstruction to the concentration phenomenon and, consequently, to the $w^*$-UKK.

\subsection*{Uncountable results}  The arguments given in this paper remain valid when the index set $\mathbb N$ is replaced by an arbitrary set $\Gamma$, as finitely supported vectors are dense, and the equality $\|x \pm t y\|_p^p = \|x\|_p^p + t^p \|y\|_p^p$ still holds for vectors with disjoint support in this context. Therefore, our results for $c_0$ and $\ell_p$ can be simply replaced immediately by $c_0(\Gamma)$ and $\ell_p(\Gamma)$.

\subsection*{Strong subdifferentiability}

Let $E$ be a Banach space and let $x\in S_E$. The norm of $E$ is said to be strongly subdifferentiable (SSD, for short) at $x$ if the one-sided limit
\[
\lim_{t\to0^+}\frac{\|x+th\|-1}{t}
\]
exists uniformly for $h\in B_E$ (see, for instance, \cite{FP, GodefroyMontesinosZizler}). We say that $E$ is \emph{SSD} if its norm is strongly subdifferentiable at every point of the unit sphere.

It is well known that if the dual space $E^*$ has the  $w^*$-KK, then the norm of $E$ is strongly subdifferentiable (see, for instance, \cite[Proposition~2.6]{DantasKimLeeMazzitelli}). Since the $w^*$-UKK property implies the $w^*$-KK, the results obtained in this paper immediately yield the following consequences.

\begin{corollary}\label{cor:ssd-consequences}
Each of the following Banach spaces has strongly subdifferentiable norm.
\begin{itemize}
\item[(a)] $\polten{2} X$, whenever $X$ is a complex Banach sequence lattice such that $c_{00}$ is dense in $X$ and $X$ is $p$-convex with constant one for some $p>2$.
\item[(b)] $X\pten Y$, whenever $X$ and $Y$ are complex Banach sequence lattices such that $c_{00}$ is dense and  $p$-convex with constant one for some $p>2$.
\item[(c)] $X\pten Y$, whenever $X$ and $Y$ are Banach sequence lattices satisfying the assumptions of \hyperref[real-theorem-3]{Theorem~\ref*{real-theorem-3}}.
\end{itemize}
\end{corollary}

Finally, these strong subdifferentiability results have an additional consequence: they can be used to obtain norm-attainment-type results for bilinear forms and $2$-homogeneous polynomials (see \cite[Proposition~4.4]{DantasKimLeeMazzitelli}).

\section{Acknowledgments} During the preparation of this manuscript, the authors used {\it ChatGPT Go} to assist with the verification of mathematical arguments, the organization of the manuscript and bibliography, language editing and the preparation of portions of the \LaTeX\ source. All AI-assisted material was subsequently reviewed, verified and revised by the authors, who take full responsibility for the content of the paper.

\section{Funding} S. Dantas has been supported by the grants PID2021-122126NB-C31 and PID2021-122126NB-C33 funded by MICIU/AEI/ 10.13039/ 501100011033 and by ERDF/EU. J. T. Rodríguez has been partially supported by CONICET PIP 11220200101609CO.

\end{document}